\documentclass[reqno]{amsart}
\usepackage{amsmath,amsthm,amssymb}
\usepackage{lipsum}
\usepackage{bm}
\usepackage[utf8]{inputenc}
\usepackage{esint}
\usepackage[all]{xy}
\usepackage{amsfonts} 
\usepackage{xcolor}
\usepackage{tikz}
\usepackage[colorlinks=true]{hyperref}
\usepackage{bbold}
\usepackage{enumerate}
\usepackage{calc}
\usepackage{bbold}
\usepackage{booktabs}
\usepackage{tikz-cd}
\usepackage[colorinlistoftodos,prependcaption,textsize=tiny]{todonotes}
\usepackage{todonotes}
\usepackage[shortcuts]{extdash}
\usepackage[style=alphabetic,maxnames=200,maxalphanames=6,giveninits]{biblatex}
\makeatletter
\numberwithin{equation}{section}
\theoremstyle{plain}
        \newtheorem{theorem}{Theorem}[section]

        \newtheorem{proposition}[theorem]{Proposition}
        \newtheorem{lemma}[theorem]{Lemma}

        \newtheorem{remark}[theorem]{Remark}

\newtheorem*{theorem*}{Theorem}
\newtheorem*{definition*}{Definition}
\newtheorem*{proposition*}{Proposition}

\newcommand{\R}{\mathbb{R}}

\newcommand{\cB}{\mathcal{B}}
\newcommand{\cR}{\mathcal{R}}

\newcommand{\ong}{\mathop{\overline{\bm{\nabla}}}\nolimits}

\newcommand{\onc}{{\overline\nabla}^c}
\newcommand{\on}{{\overline\nabla}}
\newcommand{\wn}{\widetilde\nabla}
\newcommand{\wnc}{{\widetilde\nabla}^c}

\newcommand{\cD}{\mathcal{D}}
\newcommand{\cH}{\mathcal{H}}
\newcommand{\Do}{\R^{2d}}

\def\aS{\mathsf S}
\def\aZ{\mathsf Z}
\def\az{\mathsf z}

\def\aJ{\mathsf J}
\def\av{\mathsf v}

\def\cS{\mathcal S}
\def\aR{\mathsf R}

\def\aB{\mathsf B}
\def\aL{\mathsf L}

\def\cL{\mathcal{L}}

\def\aL{\mathsf L}
\def\aM{\mathsf M}
\def\aaE{\mathsf E}

\usepackage{comment}

\newcommand{\G}{\R^{3d}}
\newcommand\supp{\operatorname{Supp}}
\renewcommand{\d}{\partial} 
\newcommand{\defeq }{\mathop{=}\limits^{\textrm{def}}}
\newcommand{\dd}{{\,\rm d}}

\definecolor{gyellow}{RGB}{228, 155, 15}
\definecolor{dgreen}{RGB}{0, 100, 0}
\definecolor{bblue}{RGB}{137, 207, 240}

\newcommand {\ZH}[1]{\todo[inline,size=\footnotesize,color=gyellow]{\textbf{ZH:} #1}}

\title[GENERIC of kinetic models]{Unified Formulation and Asymptotic Limits of Inhomogeneous Kinetic Models within GENERIC}

\author[H.~Duong]{Hong Duong}
\author[Z.~He]{Zihui He}
\address[H.~Duong]
{School of Mathematics, University of Birmingham, UK}
\email{h.duong@bham.ac.uk}
\address[Z.~He]
{Fakult\"at f\"ur Mathematik, Universit\"at Bielefeld, Postfach 100131, 33501 Bielefeld, Germany}
\email{zihui.he@uni-bielefeld.de}

\date{\today}

\begin{document}
\begin{abstract}
In this paper, we study a general class of inhomogeneous kinetic models that unifies fundamental models in both the statistical physics of particles
and of waves, namely 
the kinetic Boltzmann equations and the kinetic wave equations, in both classical (non-relativistic), relativistic and quantum settings. We formulate this unified equation into the GENERIC (General Equation for Non-Equilibrium Reversible-Irreversible Coupling) framework. We then derive the grazing (small-angle) limit in two-body interaction systems, which leads to Landau-type equations.  Finally, we show that these limiting systems can also be formulated as GENERIC systems.


\end{abstract}

\maketitle

\section{Introduction}
\subsection{Inhomogeneous kinetic equations} 
In this paper, we consider a general class of Boltzmann-type inhomogeneous kinetic models of the following form:
\begin{equation}
\label{intro:eq-uni}
\left\{
\begin{aligned}
&\d_t f+\nabla_p e(p)\cdot \nabla_q f=Q(f),\\
&Q(f)=\frac{1}{ n!}\sum_{S_n}\int_{\R^{(2n-1)d}}\delta \cB \big(\Pi_{i=0}^{n-1}\gamma_i(f_{\tau(i)}')\overline{\gamma}_i (f_{\tau(i)})-\Pi_{i=0}^{n-1}\overline{\gamma}_i (f_{\tau(i)}')\gamma_i(f_{\tau(i)})\big)\dd \eta^{2n-1}.
\end{aligned}
\right.
\end{equation}
As will be shown below, equation \eqref{intro:eq-uni} unifies fundamental models in both the statistical physics of \textit{particles}
and of \textit{waves}, namely 
the kinetic Boltzmann equations and  the kinetic wave equations, in both classical (non-relativistic), relativistic and quantum settings. In the (classical, relativistic, quantum) Boltzmann equations, the  unknown $f=f(t,q,p)$ denotes the probability distribution of the (classical, relativistic, quantum) particles in the phase space at time $t$ with position $q\in \R^d$ and momentum $p\in\R^d$. On the other hand, in the kinetic wave equations, $f$ describes the wave action density (or occupation number) at time $t$, position $q$ and wavevector $p$. The dynamics \eqref{intro:eq-uni} consists of two components: a transport term and a collision term.

The linear transport
term $\nabla_p e(p)\cdot\nabla_q f$  describes the advection of the density. In this paper, we focus on the energy functions, $e=e(p)$, associated with classical Newtonian and relativistic dynamics, respectively,
\begin{align}
\label{intro-e}
e(p)=\frac{|p|^2}{2m}\quad\text{and}\quad e(p)=c\sqrt{(mc)^2+|p|^2},
\end{align}
where $m>0$ be particle mass and $c>0$ be the speed of light. In Section \ref{sub-sec:B-gen}, different energy functions are also allowed. 

The collision operator $Q=Q(f)(q,p)$ describes the variation of the number of particles/wave with position $q$ and momentum $p$, in a unit of time, due to
collisions (interactions) between $n$ particles (in particle models) or $2n$ waves (in wave models). It is obtained as the difference between the
gain and loss terms from the interactions. We now explain the precise form these terms.

Let $p_i$ and $p_i'\in \R^d$, $i=0,\dots, n-1$ denote the input and output momenta. We write 
\begin{align*}
p_0=p,~ p_0'=p',~ f_0=f(q,p_0),~ f_0'=f'(q,p_0'),~
f_i=f(q,p_i)~\text{and}~ f_i'=f(q,p_i').
\end{align*}
Let $e_i=e(p_i)$ and $e_i'=e(p_i')$. For a single interaction, the following momentum and energy conservation laws hold
\begin{align}
\label{cl}
\sum_{i=0}^{n-1}p_i=\sum_{i=0}^{n-1}p_i'\quad\text{and}\quad \sum_{i=0}^{n-1}e_i=\sum_{i=0}^{n-1}e_i'.
\end{align} 
Let $\delta^d_0$ denote the $d$-dimensional Dirac measure. The Dirac measure $\delta$ appearing in the collision operator $Q$ is defined as follows 
\begin{align*}
    \delta\defeq\delta^1_0\Big(\sum_{i=0}^{n-1}(e_i-e_i')\Big)\delta^d_0\Big(\sum_{i=0}^{n-1}(p_i-p_i')\Big).
\end{align*}
This formally enforces the conservation laws \eqref{cl} and will be made rigorous in Section \ref{sec:Boltzmann}.

For $i=0,\dots,n-1$, we take $a_i,\overline{a}_i\in\{0,1\}$ and $\alpha_i,\overline{\alpha}_i\in\{-1,0,1\}$, and define the following functions
\begin{align}
\label{gamma-i}
   \gamma_i(f)=a_i+\alpha_if\quad\text{and}\quad \overline{\gamma}_i (f)=\overline{a}_i+\overline{\alpha}_if.
\end{align}
The specific values of the parameters $a_i, \overline{a}_i, \alpha_i,\overline{\alpha}_i$ determine the specific system and are specified explicitly in Table \ref{tab:models}.

For a function $G=G\big(p_0,\cdots,p_{n-1},p_0',\cdots,p_{n-1}'\big)$, we use the following notations to denote the transformation that swaps the group of unknowns $(p_0,\dots,p_{n-1})$ and $(p_0',\dots,p_{n-1}')$ 
\[
G'=G\big(p_0',\cdots,p_{n-1}',p_0,\cdots,p_{n-1}\big).
\]
Let $\tau\in S_n$ denote a permutation of ${0,1,\dots,n-1}$. 
The kernel  $\cB:\R^{2nd}\to\R_+$ in the collision operator $Q$ is invariant under the following transformations 
\begin{equation}
\label{def:cB}
\begin{aligned}
&\cB(p,\dots,p_{n-1},p',\dots,p_{n-1}')\\
    =&{}\cB(p',\dots,p_{n-1}',p,\dots,p_{n-1})\\
    =&{}\cB(p_{\tau(0)},\dots,p_{\tau(n-1)},p_{\tau(0)}',\dots,p_{\tau(n-1)}')\quad\forall \tau\in S_n.
\end{aligned}
\end{equation}
Moreover, we assume that $\cB$ is Galilean invariant in the classical (non-relativistic) models and is Lorentz invariant in the relativistic ones. The details of Lorentz transinformation can be found in Appendix \ref{app:Lorentz}.
Let $\tau_k$ denote the permutations on $\{0,1,\dots,n-1\}$ that only swaps $0$ and $k$. We note that, by change of variables, the collision operator $Q(f)$ can also be expressed as
 \begin{equation}
 \label{andere}
\begin{aligned}
Q(f)&=\frac{1}{ n}\sum_{k=0}^{n-1}\int_{\R^{(2n-1)d}}\delta \cB \big(\Pi_{i=0}^{n-1}\gamma_i(f_{\tau_k(i)}')\overline{\gamma}_i (f_{\tau_k(i)})-\Pi_{i=0}^{n-1}\overline{\gamma}_i (f_{\tau_k(i)}')\gamma_i(f_{\tau_k(i)})\big)\,\dd\eta^{2n-1},
\end{aligned}
\end{equation}
where $\dd\eta^{2n-1}=\dd p_1\dots \dd p_{n-1}\dd p_0'\dots \dd p_{n-1}'$ denotes the Lebesgue measure on $\R^{(2n-1)d}$. 

In the subsequent analysis, it is more convenient to group the general model into three sub-classes: (quantum) Boltzmann equations, kinetic wave equations, and linear Boltzmann equations. Their collision operators are respectively given by
\begin{align}
Q_\alpha(f)&=\int_{\R^{(2n-1)d}}\delta \cB \big(\Pi_{i=0}^{n-1}f_i'(1+\alpha f_i)-\Pi_{i=0}^{n-1}(1+\alpha f_i')f_i\big),\label{intro:Q-phonon}\\
Q_{wave}(f)&=\int_{\R^{(2n-1)d}}\delta \cB \big(\Pi_{i=0}^{n-1}f_i'f_i\big)\Big(\sum_{i=0}^{n-1}(f_i)^{-1}- (f_i')^{-1}\Big),\label{intro:Q-wave}\\
Q_{linear}(f)&=\int_{\R^{(2n-1)d}}\delta \cB \sum_{i=0}^n(f_i'-f_i).
\end{align}
Notice that the collision operators are in the form of \eqref{intro:eq-uni} up to a multiplicity constant. 
The parameter $\alpha$ in the collision $Q_\alpha$ encodes the type of statistics
\[
\alpha=\begin{cases}
    1, \quad \text{Bose-Einstein statistics}\\
    -1\quad\text{Fermi-Dirac statistics}\\
    0\quad \text{Boltzmann-Maxwell statistics}.
\end{cases}
\]
The case of $\alpha=\pm1$ is also known as Uehling--Uhlenbeck equation \cite{UU33} and the Boltzmann-Nordheim equation \cite{Nor28}. The $\alpha=+1$ is also known as the (four-) phonon equation, see \cite{spohn2006phonon}. The kinetic wave equations are central equations in the theory of wave turbulence. In recent years there have been significant breakthroughs in the rigorous derivation of the wave kinetic equations from the nonlinear Schrödinger \cite{DH21,DH23,DH23b,buckmaster2021onset,hani2024inhomogeneous}. We also refer to \cite{zakharov2012kolmogorov,nazarenko2011wave} for a detailed exposition of the wave turbulence theory.

The most popular models studied in the literature are those for collisions between two particles, that is $n=2$.  In Sections \ref{sec:Boltzmann} and \ref{sec:Landau}, we will focus on these 2-body models, in which we rigorously make sense of the Dirac-measure that appears in the collision operators. Kinetic models that involve collisions of more than 2 particles have also been studied by several authors, see for instance \cite{Cer88,Ma83} for the $n$-body interaction Boltzmann equation and \cite{pavlovic2026inhomogeneous} for the 6-wave kinetic equation.

\begin{table}
\centering
\renewcommand{\arraystretch}{1.3}
\begin{tabular}{@{}lccccc@{}}
\toprule
Models (classical/relativistic) & $(a_0,\alpha_0)$ &  $(a_i,\alpha_i)$&$(\overline{a}_0,\overline{\alpha}_0)$ & $(\overline{a}_i,\overline{\alpha}_i)$  \\
 &  & $i=1,\dots,n-1$ & & $i=1,\dots,n-1$ \\
\midrule
 Quantum kinetic (Bose) & $(0,1)$ & $(0,1)$ & $(1,1)$ & $(1,1)$  \\
 Boltzmann & $(0,1)$ & $(0,1)$ & $(1,0)$ & $(1,0)$  \\
 Quantum kinetic  (Fermi)  & $(0,1)$ & $(0,1)$ & $(1,-1)$ & $(1,-1)$  \\
 Wave kinetic &  $(0,1)$ & $(0,1)$ & $(1,0)$ & $(0,1)$  \\
Linear Boltzmann  & $(0,1)$ & $(1,0)$ & $(1,0)$ & $(1,0)$ \\
\bottomrule
\end{tabular}
\caption{$n$-body Boltzmann type of equations}
\label{tab:models}
\end{table}

\subsection{The GENERIC formalism}

The GENERIC (General Equation for Non-Equilibrium Reversible-Irreversible Coupling) framework, introduced in \cite{GO97a,GO97b},  provides a systematic approach to modelling the dynamics of nonequilibrium systems by unifying reversible and irreversible processes. 
A GENERIC system describes the evolution of an unknown $\mathsf z$ in a state space $\mathsf Z$ via the equation
\begin{equation}\label{eq:GENERIC}
\partial_t \az = \aL \dd\aaE+\aM\dd \aS.
\end{equation}
In this equation, the functionals $\aaE,\aS:\aZ\to \R$ are energy and entropy functionals  respectively, and $\dd\aaE$, $\dd\aS$ are their differentials. For each $\az\in \aZ$,   $\aL(\az)$ is an antisymmetric operator satisfying the Jacobi identity (Poisson operator), while $\aM(\az)$, $\az\in \aZ$ is a symmetric and positive semi-definite operator (Onsager or dissipative operator). In addition, the following degeneracy (orthogonality) conditions are satisfied:
\begin{equation}\label{eq:degen}
\aL(\az) \dd \aS(\az)=0\quad\text{and}\quad \aM(\az) \dd\aaE(\az) = 0\quad\text{for all }\az.
\end{equation}
Note that pure Hamiltonian systems and pure (dissipative) gradient flow systems are special cases of GENERIC corresponding to $\aM\equiv 0$ and $\aL\equiv 0$, respectively. 

The conditions satisfied by the building blocks $\{\aaE, \aS, \aL,\aM\}$ ensure that along any solution to \eqref{eq:GENERIC}, the energy $\aaE$ is conserved and the entropy $\aS$ is non-decreasing. In fact, 
\begin{align*}
\frac{\dd}{\dd t} \aaE(\mathsf z_t) = \partial_t z\cdot \dd\aaE=(\aL \dd\aaE+\aM \dd\aS)\cdot \dd\aaE=\underbrace{\aL \dd\aaE\cdot \dd \aaE}_{=0}+\underbrace{\aM \dd\aS \cdot \dd \aaE}_{=0}=0,
\end{align*}
where we have used the anti-symmetry of $\aL$, and the symmetry of $\aM$ together with the second orthogonality condition. By a similar computations, we have
$$
\frac{\dd}{\dd t} \aS(\az_t) = \dd \aS\cdot \aM \dd\aS \ge0.$$
Thus, the first and second laws of thermodynamics are automatically fulfilled for GENERIC systems. 

The GENERIC system \eqref{eq:GENERIC} can be
extended to a generalized (non-quadratic) GENERIC system \cite{mielke2011formulation}
\begin{equation}
\label{GENERIC-R}
 \d_t \az=\aL\dd\aaE+\d_\xi \aR^*(\dd\aS), 
\end{equation}
where the irreversible part  $\aM\dd \aS$ in \eqref{eq:GENERIC} is replaced by $\d \aR^*(\dd\aS)$. Here $\aR^*$ is a dissipation potential, which is a convex, superlinear and even function. When $\aR^*$ is a quadratic function, $\aR^*(\xi))=\frac{1}{2}\xi^T\aM(\az) \xi$, we recover \eqref{eq:GENERIC}.

In summary, the GENERIC framework has been successfully applied across a wide range of classical, mesoscopic, and complex systems, including fluids, polymers, soft matter, and chemical reactions, and provides a versatile framework for extending kinetic and continuum descriptions while maintaining the fundamental structure of nonequilibrium thermodynamics. We refer the reader to the book \cite{Ott05} and a recent survey \cite{grmela2018generic} for an exposition of GENERIC.

\subsection{Main results of the paper} 
The aim of this paper is threefold. Firstly, we bring the two topics discussed in the previous subsections together by formulating the general class of kinetic models \eqref{intro:eq-uni} into the GENERIC framework \eqref{eq:GENERIC}, thus shedding light on the physical/thermodynamics and geometrical structure of the former. We explicitly construct the building blocks (the energy and entropy functionals, as well as the Poisson and the Onsager operators) for the unified system. Secondly, we perform a small-angle scattering limit of \eqref{intro:eq-uni} to obtain a general unified Landau-type equation. This extends the celebrated grazing limit from the classical Boltzmann equation to the classical Landau equation. Thirdly, we show that the resulting limiting systems also exhibit GENERIC structure, thus putting all of the models in the same GENERIC framework. Below we compare the present paper with existing works in these three topics.  
\subsection*{Related works}
As already mentioned, equation \eqref{intro:eq-uni} covers many fundamental models, including the kinetic Boltzmann equations and  the kinetic wave equations in both classical (non-relativistic), relativistic and quantum settings. There are huge literature on these equations in both mathematics and physics literature. We refer the reader to the monographs \cite{Cer88, cercignani2002relativistic,villani2002review} for more information about classical kinetic theory and to \cite{zakharov2012kolmogorov,nazarenko2011wave} for wave kinetic equations and wave turbulence theory. In the below we review papers that are directly relevant to our work, either on GENERIC/gradient flow formulation or on the aspect of unifying the models.

\textit{On GENERIC formulation of kinetic models}. Our present work is motivated by \cite{Ott97, grmela2018generic}, which casts the classical kinetic Boltzmann equation into the GENERIC framework, and by our recent work \cite{duong2025generic} in which we study the GENERIC structure and small-angle limits for the classical 3-wave and 4-wave kinetic equations. In recent years, there has been a considerable progress on rigorously proving the GENERIC/gradient flow structure for classical kinetic models, see \cite{Erb23,carrillo2024landau} for the spatially homogeneous Boltzmann and Landau equations and \cite{EH25,DuongHe2025} for the corresponding fuzzy models. The present work generalises \cite{Ott97, grmela2018generic,duong2025generic} by formally formulating the classical, relativistic and quantum kinetic models, as well as the corresponding limiting systems under the grazing (small-scattering limit, see the next point) into the GENERIC framework using the unified form \eqref{intro:eq-uni}. In particular, we extends the compatibility condition in \cite{PRST20,Erb23,DGH25}, see \eqref{intro-comptb} below, that enables us to 
recast \eqref{intro:eq-uni} into the GENERIC form \eqref{eq:GENERIC}. 

\textit{On the grazing (small-scattering) limit.} The grazing limit from the Boltzmann to the Landau equation has been studied extensively by many authors, see for instance \cite{villani1998new,alexandre2004landau,carrillo2022boltzmann} for the classical Boltzmann equation, \cite{Belyaev1956,HJ24} for the relativistic Boltzmann one, and \cite{DGH25} for the fuzzy Boltzmann equation. The semi-classical limit from quantum Boltzmann equations to quantum Landau equations in the non-relativistic setting has also been studied in the literature, see for instance \cite{he2021semi, he2024semi, gualdani2025quantum}. In this paper, we perform this limit in a unified manner via the unified equation \eqref{intro:eq-uni}. In particular, as a consequence, to the best of our knowledge, the derivation of the small-angle limit in the relativistic quantum Boltzmann equations in Section \ref{sec:Landau} of this paper is new. 

\textit{On unified treatments of various kinetic models.} There exists several papers that treat various kinetic models in a unified manner. The most relevant papers to us include \cite{spohn2006phonon,escobedo2003homogeneous,EGLM25}, in which \cite{spohn2006phonon,EGLM25} studies the phonon Boltzmann equation and kinetic wave equations while \cite{escobedo2003homogeneous} investigates quantum, relativistic or non-relativistic, Boltzmann equations. However, although the conservation of energy and entropy dissipation have been discussed, these papers do not reveal their variational GENERIC structures (in particular, the construction of the dissipative operators) as in this paper. 
\subsection{Organisation of the paper}

In Section \ref{sec:Boltzmann}, we focus on two-body interaction Boltzmann-type of equations. We summarise the parametrisations that respect the momentum and energy conservation laws \eqref{cl} in both non-relativistic and relativistic cases. Moreover, we redefine the discrete gradient to incorporate these conservation laws, which leads to additional GENERIC building blocks compared to \eqref{EL}-\eqref{MLL}. 

In Section \ref{sec:Landau}, we study the small-angle (grazing) limit of both the non-relativistic and relativistic Boltzmann-type equations discussed in Section \ref{sec:Boltzmann}. We derive the resulting Landau-type equations, and construct their GENERIC building blocks.

\section{Compatibility condition and GENERIC formulation of \texorpdfstring{\eqref{intro:eq-uni}}{}}
\label{sub-sec:B-gen}
In this section, we formulate \eqref{intro:eq-uni} into the GENERIC framework \eqref{eq:GENERIC} by explicitly constructing the building blocks $\aaE,\aS,\aL,\aM$, where we take $\aaE,\aS$ to be the physically relevant energy and entropy associated to each system and the reversible part $\aL\dd\aaE$ corresponds to the transport part of \eqref{intro:eq-uni}. The major challenge is to construct the dissipative operator $\aM$. To this end, we will establish a compatibility condition that enables the reformulation of \eqref{intro:eq-uni} into the form of \eqref{eq:GENERIC}.
\subsection{Compatibility condition}
We define a discrete gradient operator for $\phi=\phi(q,p)$ as follows
\begin{align}
\label{def:grad-free}
    \ong \phi(q,p,p_1,\ldots, p_{n-1}, p_0',p_1',\ldots, p'_{n-1})=\sum_{i=0}^{n-1}\phi_i'-\phi_i,
\end{align}
where we recall the notations that $\phi_i'=\phi(q,p_i'), ~\phi_i=\phi(q,p_i)$. We define the associated discrete divergence operator, $\ong\cdot G$,  for any $G=G(q,p,\dots,p_{n-1},p_0',\dots,p_{n-1}')$ via  the following integration by parts formula 
\begin{align*}
    \int_{\R^{(2n+1)d}}G\cdot \ong \phi\dd q\dd p\dd \eta^{2n-1}=-\int_{\R^{2d}}\ong\cdot G \phi\dd q\dd p,
\end{align*}
where $\dd\eta^{2n-1}=\dd p_1\dots \dd p_{n-1}\dd p_0'\dots \dd p_{n-1}'$ denotes the Lebesgue measure on $\R^{(2n-1)d}$.
By direct computations, it follows that $\ong\cdot G=\ong\cdot G(q,p)$ can be expressed explicitly by
\begin{equation}
\label{**}
\begin{aligned}
 \ong\cdot G(q,p)&=\frac{1}{(n-1)!}\sum_{S_n}\int_{\R^{(2n-1)d}}G\circ\tau-G'\circ \tau\\
 &=\sum_{k=0}^{n-1}\int_{\R^{(2n-1)d}}G\circ\tau_k-G'\circ \tau_k
\end{aligned}
\end{equation}
where $\tau_k$ is given as in \eqref{andere}. In the above, we define
\begin{align*}
    G\circ\tau:=G\big(p_{\tau(0)},\dots,p_{\tau(n-1)},p_{\tau(0)}',\dots,p_{\tau(n-1)}'\big).
\end{align*}
We also consider a weight function $\Theta:\R^{2n}\to\R_+$ which is a $1$-homogeneous concave function. Moreover, $\Theta$ is assumed to be invariant under the transformations
\begin{align*}
 &\Theta(f,\dots,f_{n-1},f',\dots,f_{n-1}')\\
 =&{}\Theta(f',\dots,f_{n-1}',f,\dots,f_{n-1})\\
 =&{}\Theta(f_{\tau(0))},\dots,f_{\tau(n-1))},f_{\tau(0))}',\dots,f_{\tau(n-1))}')\quad\forall \tau\in S_{n}.
\end{align*}
For a simplicity of notations, we write in short-hand $ \Theta(f,\dots,f_{n-1},f',\dots,f_{n-1}')=\Theta(f)$.

Let $\Psi^*\in C^\infty(\R;\R_+)$ be a convex, superlinear,  even, and $\Psi^*(0)=0$. We define a dissipation potential $\cR^*$ by 
\begin{gather*}
\cR^*(f,\av)=\int_{\R^{(2n-1)d}}\Psi^*(\ong \av)\Theta(f)\cB \delta \dd\eta^{2n-1} .
\end{gather*}
Formally, the Gateaux derivative, $\d \cR^*:=\d_\av \cR^*$, is given by \begin{align}
\label{def:dd-aR}
    \d_\av \cR^*(f,\av)=-\ong\cdot\big(\cB\delta \Theta(f) (\Psi^*)'(\ong \av)\big).
\end{align}
Indeed, for any $\phi\in C^\infty_c$, we have 
\begin{align*}
 \langle \d_\av \cR^*(f,\av),\phi\rangle
 =&{}
\lim_{\varepsilon\to0}\varepsilon^{-1}\Big(\cR^*(f,\av+\varepsilon\phi)-\cR^*(f,\av)\Big)\\
=&{}\int_{\R^{2nd}}\ong\phi\cB \delta\Theta(f)\big(\Psi^*\big)'(\ong \av) \\ 
=&{}-\int_{\R^{2d}}\phi\ong\cdot\Big(\cB \delta\Theta(f)\big(\Psi^*\big)'(\ong \av)\Big). 
\end{align*}
Let $h\in C^1(\R)$. We say $\big(\gamma_i,\overline{\gamma}_i ,\kappa,\Psi^*\Theta,h\big)$ are compatible if the following compatibility condition holds
\begin{equation}
\label{intro-comptb}
        n (\Psi^*)' \big(\ong h'(f)\big) \Theta (f)=\sum_{k=0}^{n-1}   \Pi_{i=0}^{n-1}\gamma_i(f_{\tau_k(i)}')\overline{\gamma}_i (f_{\tau_k(i)})-\Pi_{i=0}^{n-1}\overline{\gamma}_i (f_{\tau_k(i)}')\gamma_i(f_{\tau_k(i)}),
\end{equation}
where $\tau_k$ is given as in \eqref{andere} denoting the permutation on $\{0,1,\dots,n-1\}$ that only swaps $0$ and $k$.
This extends the compatibility condition in \cite{PRST20,Erb23,DGH25}, which is introduced for 2-body interacting collision operators including the classical Boltzmann equation and its fuzzy counterpart. The dissipation potential $\Psi^*$, the weighted function $\Theta$ and the entropy density $h$ that satisfy the above compatibility condition for the corresponding models are detailed in Table \ref{tab:comptb}. Note that in Table \ref{tab:comptb},  $\cL(s,t)=\frac{s-t}{\log s-\log t}$ denotes the logarithm mean of $s,t>0$. 

Under the compatibility condition \eqref{intro-comptb}, the equation \eqref{intro:eq-uni} can be written as 
\begin{equation}
\label{uni-eq}
\begin{aligned}
\d_t f+\nabla_p e(p)\cdot \nabla_q f=-\frac{1}{2n} \d \aR^*\big(f,h'(f)\big),
\end{aligned}
\end{equation}
since, by definition of the collision operator 
\begin{equation*}
\begin{aligned}
Q(f)
&\overset{\eqref{andere}}{=}\frac{1}{n}\sum_{k=0}^{n-1}\int_{\R^{(2n-1)d}}\delta \cB \big(\Pi_{i=0}^{n-1}\gamma_i(f_{\tau_k(i)}')\overline{\gamma}_i (f_{\tau_k(i)})-\Pi_{i=0}^{n-1}\overline{\gamma}_i (f_{\tau_k(i)}')\gamma_i(f_{\tau_k(i)})\big)\\
&\overset{\eqref{**}}{=}\frac{1}{2n^2}\ong\cdot\Big(\sum_{k=0}^{n-1}\delta \cB \big(\Pi_{i=0}^{n-1}\gamma_i(f_{\tau_k(i)}')\overline{\gamma}_i (f_{\tau_k(i)})-\Pi_{i=0}^{n-1}\overline{\gamma}_i (f_{\tau_k(i)}')\gamma_i(f_{\tau_k(i)})\big)\Big)
\\&\overset{\eqref{intro-comptb}}{=}\frac{1}{2n}\ong\cdot\Big(\delta\cB(\Psi^*)' \big(\ong h'(f)\big) \Theta (f)\Big)\\
&\overset{\eqref{def:dd-aR}}{=}-\frac{1}{2n} \d \aR^*\big(f,h'(f)\big).
\end{aligned}
\end{equation*}
The equation \eqref{uni-eq} has the following weak formulation
\begin{align*}
&\int_{\Do}\varphi_0f_0\dd p\dd q-\int_0^T\int_{\Do}\big(\d_t+\nabla_p e\cdot\nabla_q\big)\varphi f\dd p \dd q\dd t\\
&=-\frac{1}{2n}\int_0^T\int_{\R^{(2n+1)d}}\delta \cB \Theta(f)\ong \phi (\Psi^*)'\big(\ong h'(f)\big)\dd q\dd p \dd\eta^{2n-1}\dd t.
\end{align*}
The equation \eqref{uni-eq} is associated with a dissipative entropy
\begin{equation}
\label{eq: entropy fcn}
    \cH(f)=\int_{\Do} h(f)\dd q\dd p\quad \text{and}\quad \dd\cH(f)=h'(f),
\end{equation}
since, at least formally, we have
\begin{align*}
\label{cH-uni}
    \frac{d}{dt}\cH(f)=-\int_{\R^{(2n+1)d}}\ong h'(f)(\Psi^*)'\big(\ong h'(f)\big)\Theta \cB\delta\dd\eta^{2n-1}\dd p\dd q\le0,
\end{align*}
where we have used the property that $\Psi^*$ is convex, non-negative and $\Psi^*(0)=0$ to get
 \begin{equation*}
 r(\Psi^*)'(r)\ge0\quad \text{for all}\quad  r\in\R.    
 \end{equation*}
It follows from the definition of the discrete gradient operator that
\begin{align*}
\delta\ong (1,p,e)=0.    
\end{align*}
As a consequence, the following mass, momentum and energy conservation laws hold, at least formally,
\begin{equation*}
\label{cl-eq}
\int_{\Do} f_t (1,p,e)   \dd p\dd q=\int_{\Do} f_0 (1,p,e)   \dd p\dd q\quad\forall t\in[0,T].
\end{equation*}
\subsection{GENERIC structure}
Equation \eqref{uni-eq} can be recast into the GENERIC framework, with the GENERIC building block $\{\aaE,\,\aS,\,\aL,\,\d\aR^*\}$ are constructed as follows. The energy and entropy functionals are respectively given by 
\begin{equation}
    \label{EL}
\begin{aligned}
&\aaE(f)=\int_{\R^{2d}} e(p) f\dd p\dd q\quad\text{and}\quad \aS(f)=-\cH(f).
\end{aligned}
\end{equation}
The operators $\aL$ and $\d\aR^*$ at $f\in\aZ$ by
\begin{equation}
    \label{MLL}
\begin{aligned}
&\aL(f)\xi=-\nabla\cdot(f \aJ\nabla \xi),\quad \aJ=\begin{pmatrix}
    0& \mathsf{id}_d\\
    -\mathsf{id}_d&0
\end{pmatrix},\\
    &\d \cR^*(f,\xi)=-\frac{1}{2n}\ong \cdot\big( \cB\delta \Theta(f) (\Psi^*)'(\ong  \xi)\big).
\end{aligned}
\end{equation}
for all $\xi\in \aZ$, where $\nabla =(\nabla_q,\nabla_{ p})^T$ denotes the traditional gradient operator. We consider the phase space $\aZ$ to be appropriated functional space endowed with the $L^2$-inner product $\langle f,g\rangle=\int_{\R^{2d}}fg\dd v\dd x$. The admissible triples $(\Psi^*, \Theta, h)$, that satisfy the compatibility condition \eqref{intro-comptb}, are shown in Table \ref{tab:comptb}. By direct calculations, one can check that the building blocks \eqref{EL}–\eqref{MLL} lead to the GENERIC system \eqref{uni-eq}. Moreover, in the quadratic case $\Psi^*(r)=\frac{r^2}{2}$, the degeneracy condition \eqref{eq:degen} holds as a consequence of the antisymmetric structure of $\aL$ and the energy conservation law \eqref{cl}.

\begin{table}
\centering
\renewcommand{\arraystretch}{1.3}
\begin{tabular}{@{}lccccc@{}}
\toprule
Models & $\Psi^*(r)$ & $\Theta(f)$ &  $h'(f)$  \\
\midrule
  (Quantum) Boltzmann & $4(\cosh(r/2)-1)$   & $\Pi_{i=0}^{n-1}\sqrt{f_i}$ & $\log \frac{f}{1+\alpha f}$  \\
  & $r^2/2$   & $\cL\Big(\Pi_{i=0}^{n-1}f'_i(1+\alpha f_i),\Pi_{i=0}^{n-1}f_i(1+\alpha f_i')\Big)$ & $\log \frac{f}{1+\alpha f}$   \\
 Wave kinetic & $r^2/2$ &  $\Pi_{i=0}^{n-1}f_if_i'$ & $-f^{-1}$   \\
Linear Boltzmann  & $r^2/2$ & $1$ & $f$  \\
\bottomrule
\end{tabular}
\caption{Compatibility conditions for Boltzmann type of equations}
\label{tab:comptb}
\end{table}
In Table \ref{tab:comptb}, the entropy density for the (quantum) Boltzmann, wave kinetic and linear Boltzmann equations are given respectively by 
\begin{equation}
\label{H-all}
\begin{gathered}
h_\alpha(f)=\left\{ \begin{aligned} 
             &f\log f - f \quad  (\text{Maxwell}) \\ 
             &f\log f - (1+f)\log(1+f) \quad  (\text{Bose})\\ 
             &f\log f + (1-f)\log(1-f) \quad  (\text{Fermi}), \end{aligned}
             \right.\\
       h_{wave}(f)=-\log f\quad\text{and}\quad  h_{linear}(f)=\frac{f^2}{2}.
\end{gathered}
       \end{equation}
In the Fermi case, we take $h_{-1}(f)=+\infty$ in the case of $f\notin[0,1]$. We recall that $\cH(f)$ is defined in \eqref{eq: entropy fcn}
in the case that $\max\big(h'(f),0\big)$ is integrable (otherwise, we take $\cH(f)=+\infty$).

As shown in Table \ref{tab:comptb}, the wave kinetic equation and the linear Boltzmann equation admits a quadratic GENERIC formulation, correspond respectively, to
\begin{align*}
&\text{WKE}:\quad \Psi^*(r)=r^2/2,~ \Theta(f)=\prod_{i=0}^{n-1}f_if_i',~h'(f)=-\frac{1}{f},\\
&\text{Linear Boltzmann}:\quad \Psi^*(r)=r^2/2,~ \Theta(f)=1,~h'(f)=f.    
\end{align*}
However, it is interesting to note that the (quantum) Boltzmann equations can be written as both quadratic and non-quadratic (more precisely, a $\cosh$ function) GENERIC formalism, corresponding to two different admissible triples of $(\Psi^*, \Theta, h)$ (see Table \ref{tab:comptb})
\begin{align*}
& \Psi^*(r)=4\Big(\cosh(r/2)-1\Big), ~\Theta(f)=\prod_{i=0}^{n-1}\sqrt{f_i},~h'(f)=\log\frac{f}{1+\alpha f},\\
& \Psi^*(r)=r^2/2, ~\Theta(f)=\cL\Big(\Pi_{i=0}^{n-1}f'_i(1+\alpha f_i),\Pi_{i=0}^{n-1}f_i(1+\alpha f_i')\Big),~h'(f)=\log\frac{f}{1+\alpha f}.
\end{align*}
The $\cosh$-gradient flow structure for jump processes has received considerable attention in recent years due to the interesting fact that they often arise from the large deviation principle of underlying stochastic processes, see for instance for the classical Boltzmann equation \cite{leonard1995large,rezakhanlou1998large,bouchet2020boltzmann,BBBO21, bodineau2023statistical,feliachi2024dynamical,feliachi2021dynamical,bodineau2023statistical}.
We refer the readers to \cite{PRST20,peletier2023cosh,DGH25} and references therein for more detailed discussions on the non-quadratic pairs.

\subsection{Relations between the models}
\label{sub-sec:limit}

\tikzset{every picture/.style={line width=0.6pt}} 

\begin{figure}[htbp]
  \centering
\begin{tikzpicture}[x=0.6pt,y=0.6pt,yscale=-1,xscale=1]

\draw    (325.3,112.45) -- (325.82,265.41) ;
\draw [shift={(325.83,268.41)}, rotate = 269.81] [fill={rgb, 255:red, 0; green, 0; blue, 0 }  ][line width=0.08]  [draw opacity=0] (14.29,-6.86) -- (0,0) -- (14.29,6.86) -- cycle    ;
\draw    (46.4,194.03) -- (47.32,347.33) ;
\draw [shift={(47.33,350.33)}, rotate = 269.66] [fill={rgb, 255:red, 0; green, 0; blue, 0 }  ][line width=0.08]  [draw opacity=0] (14.29,-6.86) -- (0,0) -- (14.29,6.86) -- cycle    ;
\draw    (170.92,167.59) -- (350.94,166.66) ;
\draw [shift={(353.94,166.65)}, rotate = 179.7] [fill={rgb, 255:red, 0; green, 0; blue, 0 }  ][line width=0.08]  [draw opacity=0] (14.29,-6.86) -- (0,0) -- (14.29,6.86) -- cycle    ;
\draw    (137.5,152.47) -- (273.75,91.31) ;
\draw [shift={(276.49,90.09)}, rotate = 155.83] [fill={rgb, 255:red, 0; green, 0; blue, 0 }  ][line width=0.08]  [draw opacity=0] (14.29,-6.86) -- (0,0) -- (14.29,6.86) -- cycle    ;
\draw    (137.5,157.67) .. controls (142.28,160.98) and (141.74,167.59) .. (170.92,167.59) ;
\draw    (140.95,181.77) .. controls (142.81,181.3) and (145.99,169.01) .. (170.92,167.59) ;
\draw    (173.04,358.76) -- (352,358.76) ;
\draw [shift={(355,358.76)}, rotate = 180] [fill={rgb, 255:red, 0; green, 0; blue, 0 }  ][line width=0.08]  [draw opacity=0] (14.29,-6.86) -- (0,0) -- (14.29,6.86) -- cycle    ;
\draw    (139.62,343.64) -- (275.87,282.48) ;
\draw [shift={(278.61,281.25)}, rotate = 155.83] [fill={rgb, 255:red, 0; green, 0; blue, 0 }  ][line width=0.08]  [draw opacity=0] (14.29,-6.86) -- (0,0) -- (14.29,6.86) -- cycle    ;
\draw    (139.62,348.84) .. controls (144.4,352.15) and (143.87,358.76) .. (173.04,358.76) ;
\draw    (143.07,372.94) .. controls (144.93,372.47) and (148.11,360.18) .. (173.04,358.76) ;
\draw    (392,197.6) -- (392.23,204.6) -- (392.01,326.6) ;
\draw [shift={(392,329.6)}, rotate = 270.1] [fill={rgb, 255:red, 0; green, 0; blue, 0 }  ][line width=0.08]  [draw opacity=0] (14.29,-6.86) -- (0,0) -- (14.29,6.86) -- cycle    ;
\draw    (375,87) .. controls (389,110) and (418,132) .. (492,132) ;
\draw    (361,279) .. controls (375,302) and (404,324) .. (478,324) ;
\draw    (449,163) .. controls (456,151) and (477,132) .. (492,132) ;
\draw    (435,355) .. controls (442,343) and (463,324) .. (478,324) ;
\draw    (492,132) -- (545,132) ;
\draw [shift={(548,132)}, rotate = 180] [fill={rgb, 255:red, 0; green, 0; blue, 0 }  ][line width=0.08]  [draw opacity=0] (14.29,-6.86) -- (0,0) -- (14.29,6.86) -- cycle    ;
\draw    (478,324) -- (549,324) ;
\draw [shift={(552,324)}, rotate = 180] [fill={rgb, 255:red, 0; green, 0; blue, 0 }  ][line width=0.08]  [draw opacity=0] (14.29,-6.86) -- (0,0) -- (14.29,6.86) -- cycle    ;
\draw    (602,172.6) -- (601.18,274.58) -- (601.03,291.6) ;
\draw [shift={(601,294.6)}, rotate = 270.52] [fill={rgb, 255:red, 0; green, 0; blue, 0 }  ][line width=0.08]  [draw opacity=0] (14.29,-6.86) -- (0,0) -- (14.29,6.86) -- cycle    ;

\draw (325.62,89.08) node  [font=\normalsize,xscale=0.8,yscale=0.8]  {$ \begin{array}{l}
\mathrm{relativistic\ }\\
\mathrm{WKE}
\end{array}$};
\draw (48.02,168.23) node  [rotate=-359.8,xscale=0.8,yscale=0.8]  {$ \begin{array}{l}
\mathrm{relativistic\ }\\
\mathrm{quantum}\\
\mathrm{Boltzmann}
\end{array}$};
\draw (398.66,162.28) node  [rotate=-359.8,xscale=0.8,yscale=0.8]  {$ \begin{array}{l}
\mathrm{relativistic\ }\\
\mathrm{Boltzmann}
\end{array}$};
\draw (318.87,278.01) node  [rotate=-359.8,xscale=0.8,yscale=0.8]  {$\mathrm{WKE}$};
\draw (47.74,359.26) node  [rotate=-359.8,xscale=0.8,yscale=0.8]  {$ \begin{array}{l}
\mathrm{quantum}\\
\mathrm{Boltzmann}
\end{array}$};
\draw (399.59,358.87) node  [rotate=-359.8,xscale=0.8,yscale=0.8]  {$\mathrm{Boltzmann}$};
\draw (151.33,123.46) node [anchor=north west][inner sep=0.75pt]  [rotate=-335.76,xscale=0.8,yscale=0.8] [align=left] {\textcolor[rgb]{0.96,0.65,0.14}{kinetic limit }};
\draw (173.42,148.98) node [anchor=north west][inner sep=0.75pt]  [xscale=0.8,yscale=0.8] [align=left] {\textcolor[rgb]{0.29,0.56,0.89}{semi-classical limit }};
\draw (372.99,313.18) node [anchor=north west][inner sep=0.75pt]  [rotate=-270,xscale=0.8,yscale=0.8] [align=left] {\textcolor[rgb]{0.72,0.91,0.53}{Newtonian limit }};
\draw (307.12,242.48) node [anchor=north west][inner sep=0.75pt]  [rotate=-270,xscale=0.8,yscale=0.8] [align=left] {\textcolor[rgb]{0.72,0.91,0.53}{Newtonian \ limit }};
\draw (93.28,145.98) node [anchor=north west][inner sep=0.75pt]  [xscale=0.8,yscale=0.8] [align=left] {(Bose)};
\draw (93.4,170.55) node [anchor=north west][inner sep=0.75pt]  [xscale=0.8,yscale=0.8] [align=left] {(Fermi)};
\draw (165.5,312.04) node [anchor=north west][inner sep=0.75pt]  [rotate=-335.37,xscale=0.8,yscale=0.8] [align=left] {\textcolor[rgb]{0.96,0.65,0.14}{ kinetic limit }};
\draw (172.62,339.21) node [anchor=north west][inner sep=0.75pt]  [xscale=0.8,yscale=0.8] [align=left] {\textcolor[rgb]{0.29,0.56,0.89}{semi-classical limit }};
\draw (95.4,337.15) node [anchor=north west][inner sep=0.75pt]  [xscale=0.8,yscale=0.8] [align=left] {(Bose)};
\draw (95.53,361.72) node [anchor=north west][inner sep=0.75pt]  [xscale=0.8,yscale=0.8] [align=left] {(Fermi)};
\draw (604.66,129.28) node  [rotate=-359.8,xscale=0.8,yscale=0.8]  {$ \begin{array}{l}
\mathrm{relativistic}\\
\mathrm{linear\ }\\
\mathrm{Boltzmann}
\end{array}$};
\draw (607.66,329.28) node  [rotate=-359.8,xscale=0.8,yscale=0.8]  {$ \begin{array}{l}
\mathrm{linear}\\
\mathrm{Boltzmann}
\end{array}$};
\draw (457,112) node [anchor=north west][inner sep=0.75pt]  [xscale=0.8,yscale=0.8] [align=left] {\textcolor[rgb]{0.74,0.06,0.88}{linear limit}};
\draw (459,305) node [anchor=north west][inner sep=0.75pt]  [xscale=0.8,yscale=0.8] [align=left] {\textcolor[rgb]{0.74,0.06,0.88}{linear limit}};
\draw (581.99,279.18) node [anchor=north west][inner sep=0.75pt]  [rotate=-270,xscale=0.8,yscale=0.8] [align=left] {\textcolor[rgb]{0.72,0.91,0.53}{Newtonian limit }};
\draw (28.12,323.48) node [anchor=north west][inner sep=0.75pt]  [rotate=-270,xscale=0.8,yscale=0.8] [align=left] {\textcolor[rgb]{0.72,0.91,0.53}{Newtonian \ limit }};

\end{tikzpicture}

  \caption{$n$-body interaction Boltzmann type of equations}
  \label{fig:Boltzmann}
\end{figure}
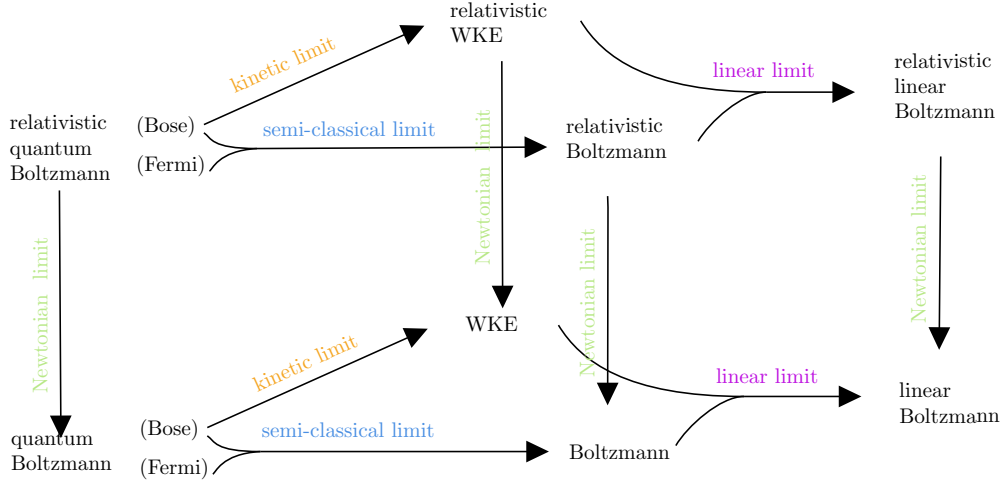

The Boltzmann-type equations listed in Table \ref{tab:models} are connected through semi-classical, kinetic, Newtonian, and linear limits. In this subsection, we review these limits, see Figure \ref{fig:Boltzmann} for an illustrative summary.



\subsection*{Newtonian limit} The non-relativistic models can be derived from the relativistic ones in the Newtonian limit, that is when the speed of light tends to infinity, see for instance \cite{strain2010global, HJ24}. In the non-relativistic and relativistic settings, the energy are respectively given by \eqref{intro-e},
\begin{align*}
e(p)=\frac{|p|^2}{2m}\quad\text{and}\quad e(p)=c p_0,\quad \text{where}\quad  p_0:=\sqrt{(mc)^2+|p|^2}.
\end{align*}
The corresponding non-relativistic and relativistic equations can be written as
\begin{align*}
\d_t f+\frac{p}{m}\cdot\nabla_q f=Q(f)\quad\text{and}\quad \d_t f+\frac{c p}{p_0}\cdot\nabla_q f=Q^c(f),
\end{align*}
where the interaction operators depend on the non-relativistic and relativistic collision kernel $\mathcal{B}$ and $\mathcal{B}^c$, respectively (see Section \ref{sec:Boltzmann} below). In the Newtonian limit, $c\to+\infty$, the relativistic transport term  $\frac{c p}{p_0}\cdot\nabla_q$ converges to the non-relativistic transport term $\frac{p}{m}\cdot\nabla_q$. The convergence of the interaction operator $Q^c(f)\to Q(f)$ will be discussed in detail, for the case $n=2$, in Section \ref{sec:limit-n=2}.
\subsection*{Semi-classical limits} Let $\hbar\in(0,1)$ be the Planck constant. We consider the following scaled quantum Bose ($\alpha=1$)/Fermi ($\alpha=-1$) equation
 \begin{equation*}
 \begin{aligned}
    &\d_t f+\nabla_pe(p)\cdot \nabla_q f\\
    =&{}\int_{\R^{(2n-1)d}}\delta \cB \big(\Pi_{i=0}^{n-1}f'_i(1+\hbar \alpha f_i)-\Pi_{i=0}^{n-1}f_i(1+\hbar \alpha f_i')\big)\dd\eta^{2n-1}
\end{aligned}
\end{equation*}
The semi-classical limit is the limit when the Planck constant tends to zero, $\hbar\to 0$. In this limit, the quantum (relativistic/non-relativistic) Boltzmann equations converge to classical (relativistic and non-relativistic) ones. 

\subsection*{Kinetic limit}
Let $\varepsilon\in(0,1)$. Let $f^\varepsilon= f(\varepsilon^{-(n-1)}t,\varepsilon^{-(n-1)}q,p)$ be a solution to the scaled Bose equation  
\begin{equation*}
 \begin{aligned}
    &\d_t f+\nabla_pe(p)\cdot \nabla_q f\\
    =&{}\int_{\R^{(2n-1)d}}\delta \cB \big(\Pi_{i=0}^{n-1}f'_i(1+\varepsilon^{-1} f_i)-\Pi_{i=0}^{n-1}f_i(1+\varepsilon^{-1} f_i')\big)\dd\eta^{2n-1}.
\end{aligned}
\end{equation*}
The kinetic limit corresponds to passing $\varepsilon\to 0$. In this limit, the quantum (relativistic/non-relativistic) Boltzmann equations to the (relativistic/non-relativistic) kinetic wave equations, see for instance~ \cite{spohn2006phonon,zakharov2012kolmogorov}. 
\subsection*{Linear limit} Let $\varepsilon\in(0,1)$. Let $f^\varepsilon= \varepsilon f(\varepsilon^{-1}t,\varepsilon^{-1}q,p)$. Let $g^\varepsilon=1+f^\varepsilon$ be a solution to the Boltzmann equation (\eqref{intro:eq-uni}-\eqref{intro:Q-phonon} with $\alpha=0$) or the wave kinetic equation \eqref{intro:eq-uni}-\eqref{intro:Q-wave}.
In the linear limit as $\varepsilon\to 0$, the perturbation equation of $f$ converges to the linear Boltzmann equations. The perturbation around the Maxwellian equilibrium was studied in the context of hydrodynamic limits for the Boltzmann equation, see for example
\cite{GSR04}. Here, instead, we consider a perturbation around 
$1$, which is permutation-invariant and therefore admits a GENERIC formulation.

\subsection{Grazing limits}
The limits discussed in the previous subsections concern the relations between Boltzmann-type equations at different physical descriptions, namely quantum, classical and relativistic settings. For two-body interaction systems, another important limit that has been studied extensively in the literature is the so-called grazing limit, that is when the angle of collisions tends to zero. In this limit, a Boltzmann-type equation converges to a corresponding Landau-type equation. These limits are summarised in Figure \ref{fig:Landau}. The detailed two-body Boltzmann-type equations will be presented in Section \ref{sec:Boltzmann}, while the small-angle limit and the resulting Landau-type equations will be discussed in Section \ref{sec:Landau}. In particular, we will show that these Landau-type equations are also GENERIC systems.



\tikzset{every picture/.style={line width=0.6pt}} 

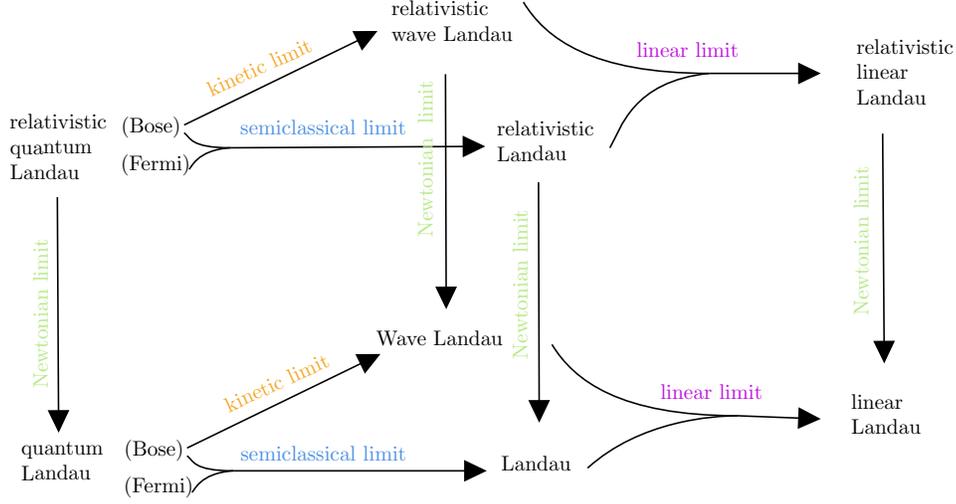
\begin{figure}
\begin{tikzpicture}[x=0.6pt,y=0.6pt,yscale=-1,xscale=1]

\draw    (313.89,83.49) -- (314.34,228.81) ;
\draw [shift={(314.35,231.81)}, rotate = 269.82] [fill={rgb, 255:red, 0; green, 0; blue, 0 }  ][line width=0.08]  [draw opacity=0] (14.29,-6.86) -- (0,0) -- (14.29,6.86) -- cycle    ;
\draw    (69.2,161.07) -- (70,306.72) ;
\draw [shift={(70.02,309.72)}, rotate = 269.68] [fill={rgb, 255:red, 0; green, 0; blue, 0 }  ][line width=0.08]  [draw opacity=0] (14.29,-6.86) -- (0,0) -- (14.29,6.86) -- cycle    ;
\draw    (178.44,129.94) -- (336,129.05) ;
\draw [shift={(339,129.04)}, rotate = 179.68] [fill={rgb, 255:red, 0; green, 0; blue, 0 }  ][line width=0.08]  [draw opacity=0] (14.29,-6.86) -- (0,0) -- (14.29,6.86) -- cycle    ;
\draw    (149.13,115.55) -- (268.36,57.54) ;
\draw [shift={(271.06,56.23)}, rotate = 154.05] [fill={rgb, 255:red, 0; green, 0; blue, 0 }  ][line width=0.08]  [draw opacity=0] (14.29,-6.86) -- (0,0) -- (14.29,6.86) -- cycle    ;
\draw    (149.13,120.5) .. controls (153.31,123.64) and (152.85,129.94) .. (178.44,129.94) ;
\draw    (152.15,143.42) .. controls (153.78,142.97) and (156.57,131.29) .. (178.44,129.94) ;
\draw    (180.31,333.74) -- (336.94,333.74) ;
\draw [shift={(339.94,333.74)}, rotate = 180] [fill={rgb, 255:red, 0; green, 0; blue, 0 }  ][line width=0.08]  [draw opacity=0] (14.29,-6.86) -- (0,0) -- (14.29,6.86) -- cycle    ;
\draw    (150.99,319.35) -- (270.22,261.34) ;
\draw [shift={(272.92,260.03)}, rotate = 154.05] [fill={rgb, 255:red, 0; green, 0; blue, 0 }  ][line width=0.08]  [draw opacity=0] (14.29,-6.86) -- (0,0) -- (14.29,6.86) -- cycle    ;
\draw    (150.99,324.3) .. controls (155.18,327.44) and (154.71,333.74) .. (180.31,333.74) ;
\draw    (154.01,347.22) .. controls (155.64,346.77) and (158.43,335.08) .. (180.31,333.74) ;
\draw    (372.54,151.67) -- (372.97,300.38) ;
\draw [shift={(372.98,303.38)}, rotate = 269.83] [fill={rgb, 255:red, 0; green, 0; blue, 0 }  ][line width=0.08]  [draw opacity=0] (14.29,-6.86) -- (0,0) -- (14.29,6.86) -- cycle    ;
\draw    (363,38) .. controls (377,61) and (406,83) .. (480,83) ;
\draw    (381,254) .. controls (395,277) and (424,299) .. (498,299) ;
\draw    (417.5,130) .. controls (424.5,118) and (431.5,88) .. (480,83) ;
\draw    (403.5,332) .. controls (436.5,302) and (483,299) .. (498,299) ;
\draw    (480,83) -- (547.5,83) ;
\draw [shift={(550.5,83)}, rotate = 180] [fill={rgb, 255:red, 0; green, 0; blue, 0 }  ][line width=0.08]  [draw opacity=0] (14.29,-6.86) -- (0,0) -- (14.29,6.86) -- cycle    ;
\draw    (498,299) -- (547.5,300.89) ;
\draw [shift={(550.5,301)}, rotate = 182.18] [fill={rgb, 255:red, 0; green, 0; blue, 0 }  ][line width=0.08]  [draw opacity=0] (14.29,-6.86) -- (0,0) -- (14.29,6.86) -- cycle    ;
\draw    (589.5,121) -- (590.48,263) ;
\draw [shift={(590.5,266)}, rotate = 269.6] [fill={rgb, 255:red, 0; green, 0; blue, 0 }  ][line width=0.08]  [draw opacity=0] (14.29,-6.86) -- (0,0) -- (14.29,6.86) -- cycle    ;

\draw (318.16,51.27) node  [font=\normalsize,xscale=0.8,yscale=0.8]  {$ \begin{array}{l}
\mathrm{relativistic\ }\\
\mathrm{wave\ Landau}
\end{array}$};
\draw (72.13,130.49) node  [rotate=-359.8,xscale=0.8,yscale=0.8]  {$ \begin{array}{l}
\mathrm{relativistic\ }\\
\mathrm{quantum}\\
\mathrm{Landau}
\end{array}$};
\draw (379.24,126.88) node  [rotate=-359.8,xscale=0.8,yscale=0.8]  {$ \begin{array}{l}
\mathrm{relativistic\ }\\
\mathrm{Landau}
\end{array}$};
\draw (310.24,249.94) node  [rotate=-359.8,xscale=0.8,yscale=0.8]  {$\mathrm{Wave\ Landau}$};
\draw (72.38,328.21) node  [rotate=-359.8,xscale=0.8,yscale=0.8]  {$ \begin{array}{l}
\mathrm{quantum}\\
\mathrm{Landau}
\end{array}$};
\draw (371.06,328.84) node  [rotate=-359.8,xscale=0.8,yscale=0.8]  {$\mathrm{Landau}$};
\draw (160.69,88.35) node [anchor=north west][inner sep=0.75pt]  [rotate=-335.76,xscale=0.8,yscale=0.8] [align=left] {\textcolor[rgb]{0.96,0.65,0.14}{ kinetic limit }};
\draw (183.22,110.97) node [anchor=north west][inner sep=0.75pt]  [xscale=0.8,yscale=0.8] [align=left] {\textcolor[rgb]{0.29,0.56,0.89}{semiclassical limit }};
\draw (52.64,283.03) node [anchor=north west][inner sep=0.75pt]  [rotate=-270,xscale=0.8,yscale=0.8] [align=left] {\textcolor[rgb]{0.72,0.91,0.53}{Newtonian limit }};
\draw (354.92,264.52) node [anchor=north west][inner sep=0.75pt]  [rotate=-270,xscale=0.8,yscale=0.8] [align=left] {\textcolor[rgb]{0.72,0.91,0.53}{Newtonian limit }};
\draw (295.14,187.55) node [anchor=north west][inner sep=0.75pt]  [rotate=-270,xscale=0.8,yscale=0.8] [align=left] {\textcolor[rgb]{0.72,0.91,0.53}{Newtonian \ limit }};
\draw (108.06,109.06) node [anchor=north west][inner sep=0.75pt]  [xscale=0.8,yscale=0.8] [align=left] {(Bose)};
\draw (107.92,132.43) node [anchor=north west][inner sep=0.75pt]  [xscale=0.8,yscale=0.8] [align=left] {(Fermi)};
\draw (171.71,287.74) node [anchor=north west][inner sep=0.75pt]  [rotate=-335.37,xscale=0.8,yscale=0.8] [align=left] {\textcolor[rgb]{0.96,0.65,0.14}{kinetic limit}};
\draw (183.39,316.72) node [anchor=north west][inner sep=0.75pt]  [xscale=0.8,yscale=0.8] [align=left] {\textcolor[rgb]{0.29,0.56,0.89}{semiclassical limit }};
\draw (109.92,312.86) node [anchor=north west][inner sep=0.75pt]  [xscale=0.8,yscale=0.8] [align=left] {(Bose)};
\draw (109.79,336.23) node [anchor=north west][inner sep=0.75pt]  [xscale=0.8,yscale=0.8] [align=left] {(Fermi)};
\draw (603.66,83.28) node  [rotate=-359.8,xscale=0.8,yscale=0.8]  {$ \begin{array}{l}
\mathrm{relativistic}\\
\mathrm{linear\ }\\
\mathrm{Landau}
\end{array}$};
\draw (591.66,299.28) node  [rotate=-359.8,xscale=0.8,yscale=0.8]  {$ \begin{array}{l}
\mathrm{linear}\\
\mathrm{Landau}
\end{array}$};
\draw (433,62) node [anchor=north west][inner sep=0.75pt]  [xscale=0.8,yscale=0.8] [align=left] {\textcolor[rgb]{0.74,0.06,0.88}{linear limit}};
\draw (448,278) node [anchor=north west][inner sep=0.75pt]  [xscale=0.8,yscale=0.8] [align=left] {\textcolor[rgb]{0.74,0.06,0.88}{linear limit}};
\draw (569.99,237.18) node [anchor=north west][inner sep=0.75pt]  [rotate=-270,xscale=0.8,yscale=0.8] [align=left] {\textcolor[rgb]{0.72,0.91,0.53}{Newtonian limit }};

\end{tikzpicture}

 \caption{Landau-type equations}
  \label{fig:Landau}
\end{figure}

\section{\texorpdfstring{$2$}{TEXT}-body interaction Boltzmann type of equations}\label{sec:Boltzmann}
In this section, we present in detail the Boltzmann-type equations shown in Figure \ref{fig:Boltzmann} for the case of two-body interactions ($n=2$). Using the notations $p,\,p_*,\,p'$ and $p_*'$ for the incoming and outgoing momenta, the general equation \eqref{intro:eq-uni} in this case becomes
\begin{equation*}
\label{eq:2-body}
\d_t f+\nabla_p e(p)\cdot \nabla_q f=Q(f),
\end{equation*}
where 
\begin{equation}
\label{eq:2-body operator}
Q(f)=\frac{1}{2}\int_{\R^{3d}}\delta\cB(p,p_*,p',p_*')(q_0(f)+q_1(f))\,dp_*dp'd p_*',  
\end{equation}
where $\delta:=\delta_0^d(e(p)+e(p_*)-e(p')-e(p_*'))\delta_0^1(p+p_*-p'-p_*') $ and
\begin{align*}
q_0(f)&=(a_0+\alpha_0 f(q,p'))(a_1+\alpha_1 f(q,p_*'))(\overline{a}_0+\overline{\alpha}_0 f(q,p))(\overline{a}_1+\overline{\alpha}_1 f(q,p_*))
\\&\qquad-(\overline{a}_0+\overline{\alpha}_0 f(q,p'))(\overline{a}_1+\overline{\alpha}_1 f(q,p_*'))(a_0+\alpha_0 f(q,p))(a_1+\alpha_1 f(q,p_*),  
\end{align*}
and 
\begin{align*}
q_1(f)&=(a_0+\alpha_0 f(q,p_*'))(a_1+\alpha_1 f(q,p'))(\overline{a}_0+\overline{\alpha}_0 f(q,p_*))(\overline{a}_1+\overline{\alpha}_1 f(q,p))
\\&\qquad-(\overline{a}_0+\overline{\alpha}_0 f(q,p_*'))(\overline{a}_1+\overline{\alpha}_1 f(q,p'))(a_0+\alpha_0 f(q,p_*))(a_1+\alpha_1 f(q,p)    
\end{align*}
In particular, the collision operators for the (quantum) Boltzmann equation, the four-wave kinetic equation, and the linear Boltzmann equation (up to a multiplicity constant) are, respectively, given  by 
\begin{align}
Q^{\aB}_{\alpha}(f)&=\int_{\R^{3d}}\delta\cB \big(f'f_*'(1+\alpha f)(1+\alpha f_*)-ff_*(1+\alpha f')(1+\alpha f_*')\big)\dd p_*\dd p'\dd p_*', \label{Q:phonon1}  \\
    Q^{\aB}_{wave}(f)&=\int_{\R^{3d}}\delta\cB\big(f'f_*'f+f'f_*'f'-ff_*f'-ff_*f_*'\big)\dd p_*\dd p'\dd p_*',\label{Q:wave1}\\
Q^{\aB}_{linear}(f)&=\int_{\R^{3d}}\delta\cB
(f'+f_*'-f-f_*)\dd p_*\dd p'\dd p_*'. \label{Q:linear1} 
\end{align}

We recall that the classical and relativistic kinetic energy are respectively given by
\begin{align}
\label{def:e,e-c}
    e(p)=\frac{|p|^2}{2m}\quad\text{and}\quad e(p)=c\sqrt{(mc)^2+|p|^2}.
\end{align}
It is known that the post-collisional momenta $p'$ and $p_*'$ can be parametrised so that the following momentum and energy conservation laws hold
\begin{align}
\label{sec4:cl}
    p+p_*=p'+p_*'\quad \text{and}\quad e+e_*=e'+e_*'.
\end{align}
With such a parametrisation, one can rigorously make sense of the Dirac measure in the definition of the collision operator $Q$ in \eqref{eq:2-body operator}, formulating it as an integral of the form $\int_{\R^d\times S^{d-1}}\dots \dd p_* \dd \omega$, where $\omega$ is the parametrisation parameter. This parametrisation also enables us to rigorously define the discrete gradients $\on$ (in the classical setting) and $\onc$ (in the relativistic setting), see \eqref{def:on} and  \eqref{def:on-c} below. These operators take into account the conservation laws \eqref{sec4:cl}. As a consequence, we also rigorously define the GENERIC building blocks associated with $\on$ and $\onc$ in the classical and relativistic cases.

In the rest of this section, we will present the parameterisation of equations and the $\on$-GENERIC building blocks in the classical and relativistic cases in Section \ref{sub-sec:Boltzmann} and Section \ref{sub-sec:Boltzmann}, respectively.

 The GENERIC structures of the classical Boltzmann, wave kinetic, and linear Boltzmann equations are already known, see for example \cite{Ottinger2018,EH25}.  The new contribution of this work is the derivation of the GENERIC building blocks for relativistic Boltzmann-type equations, presented in Section \ref{sub-sec:Boltzmann-c}.

\subsection{Non-relativistic settings}\label{sub-sec:Boltzmann}
In the classical case, we have the following parametrisation such that the momentum and energy conservation laws \eqref{sec4:cl} holds, see for instance \cite{villani1998new}
\begin{align}
\label{post-p}
 p'=\frac{p+p_*}{2}+\frac{|p-p_*|}{2}\omega,\quad p'_*=\frac{p+p_*}{2}-\frac{|p-p_*|}{2}\omega,\quad \omega\in S^{d-1}.    
\end{align}

We combine the definition of the free discrete gradient $\ong$ defined in \eqref{def:grad-free} with the above parametrisation enforcing the conservation laws to redefine the classical discrete gradient associated with $\omega$ as follows
\begin{align}
\label{def:on}
    \overline\nabla \phi=\phi'+\phi_*'-\phi-\phi_*,
\end{align}
where we write
\begin{align*}
    \phi'=\phi(q,p'),\quad \phi'_*=\phi(q,p'_*),\quad
 \phi_*=\phi(q,p_*).
 \end{align*}
For any  $\phi=\phi(q,p)$ and $G=G(q,p,p_*,p',p_*')$, the following integration by parts formula holds  
\begin{align*}
\int_{\G\times S^{d-1}}G\on\phi\dd\omega\dd p_*\dd p \dd q=-\int_{\Do}(\on\cdot G) \phi\dd p\dd q,
\end{align*}
where the divergence operator, $\on\cdot G$, is given by
\begin{align*}
\on\cdot G(q,p)=&\int_{\R^d\times S^{d-1}}G(q,p,p_*,p',p_*')+G(q,p_*,p,p_*',p')\\
&-G(q,p',p_*',p,p_*)-G(q,p_*',p',p_*,p)\dd\omega\dd p_*.
\end{align*}

We have the following lemma to evaluate the Dirac measure $\delta=\delta^0(p_0+p_{0*}-e'-e_*')\delta^d(p+p_*-p'-p_*')$ in \eqref{eq:2-body operator}.
\begin{lemma}\label{lem:cor}
Let $\cB=\cB(p,p_*,p',p_*')\ge0$ be the collision kernel in \eqref{eq:2-body operator}.
For any $G=G(q,p,p_*,w,w_*)$, we have 
\begin{align}
\label{int-delta}
    \int_{\Do} \delta \cB G(q,p,p_*,w,w_*)\dd w\dd w_*=\int_{S^{d-1}}BG(q,p,p_*,p',p_*')\dd\omega,
\end{align}
where $p'$ and $p_*'$ are given by \eqref{post-p}, and $B=B(p,p_*,\omega) \ge 0$ is the modified collision kernel such that 
\begin{equation}
\label{modified B}    
B=\frac{|p-p_*|^{d-2}}{2^d}\cB.
\end{equation}
\end{lemma}
\begin{proof}
Since $q,p,p_*$ will be fixed in this lemma, for the simplicity of notations, we write $G(q,p,p_*,w,w_*)=G(w,w_*)$.
By straightforward calculations, we have 
\begin{align}
    &\int_{\Do} \delta^d(p+p_*-w-w_*)\delta^1(|p|^2+|p_*|^2-|w|^2-|w_*|^2)G(w,w_*)\dd w\dd w_*\notag\\
    =&{}\int_{\R^d} \delta^1(|p|^2+|p_*|^2-|w|^2-|p+p_*-w|^2)G(w,p+p_*-w)\dd w.\label{eq:temp1}
\end{align}
Let $w=\frac{p+p_*}{2}+v=:\frac{p+p_*}{2}+|v|\omega$. Then we have $\dd w=|v|^{d-1}\dd |v|\dd \omega$ for $|v|\in\R_+$ and $\omega\in S^{d-1}$. Substituting these expressions into \eqref{eq:temp1} we get
\begin{align*}
    &\int_{\R^d} \delta^1(|p|^2+|p_*|^2-|w|^2-|p+p_*-w|^2)G(w,p+p_*-w)\dd w\\
    =&{}\int_{\R_+}\int_{S^{d-1}} \delta^1\big(\frac{|p-p_*|^2}{2}-2|v|^2\big)G\big(\frac{p+p_*}{2}+|v|\omega,\frac{p+p_*}{2}-|v|\omega\big)|v|^{d-1}\dd |v|\dd \omega.
\end{align*}
By using the identity 
\[
\delta^1\big(\frac{|p-p_*|^2}{2}-2|v|^2\big)=\frac{1}{2|p-p_*|}\delta^1(\frac{|p-p_*|}{2}-|v|),
\]
we have 
\begin{align*}
&\int_{S^{d-1}} \int_{\R_+}\delta^1\big(\frac{|p-p_*|^2}{2}-2|v|^2\big)G\big(\frac{p+p_*}{2}+|v|\omega,\frac{p+p_*}{2}-|v|\omega\big)|v|^{d-1}\dd |v|\dd \omega
\\&\quad=\int_{S^{d-1}}G\big(\frac{p+p_*}{2}+\frac{|p-p_*|}{2}\omega,\frac{p+p_*}{2}-\frac{|p-p_*|}{2}\omega\big)\Big(\frac{|p-p_*|}{2}\Big)^{d-1}\frac{1}{2|p-p_*|}\dd \omega
\\&\quad=\frac{|p-p_*|^{d-2}}{2^d}\int_{S^{d-1}} G(p',p_*')\dd \omega.
\end{align*}
The claimed identity \eqref{int-delta} is then followed by incorporating the kernel $\cB$, which will be transformed to the modified kernel $B$ given in \eqref{modified B} according to the above calculations.
\end{proof}
Applying Lemma \ref{lem:cor} to the collision operators \eqref{Q:phonon1}, \eqref{Q:wave1}, and \eqref{Q:linear1}, we obtain the following parametrised two-body interaction Boltzmann-type of equations
\begin{equation}
    \label{uni:Boltzmann}
\d_t f+\frac{p}{m}\cdot \nabla_q f=Q^{\aB}(f),
\end{equation}
where the (quantum) Boltzmann, four-wave kinetic, and linear Boltzmann collision operators are given by,  respectively 
\begin{align}
Q^{\aB}_{\alpha}(f)&=\int_{\R^d\times S^{d-1}}B\big(f'f_*'(1+\alpha f)(1+\alpha f_*)-ff_*(1+\alpha f')(1+\alpha f_*')\big)\dd\omega \dd p_*, \label{Q:phonon}  \\
    Q^{\aB}_{wave}(f)&=\int_{\R^d\times S^{d-1}}B\big(f'f_*'f+f'f_*'f'-ff_*f'-ff_*f_*'\big)\dd\omega \dd p_*,\label{Q:wave}\\
Q^{\aB}_{linear}(f)&=\int_{\R^d\times S^{d-1}}B
(f'+f_*'-f-f_*)\dd\omega \dd p_*. \label{Q:linear} 
\end{align}
In the classical case, we take the kernel of the following form 
\begin{equation}
\label{def:B}
    B=B(|p-p_*|,\omega)=\sigma(|p-p_*|)b(\theta)\ge 0,
\end{equation}
where $\sigma,\,b:\R_+\to\R_+$ are smooth functions, and $\theta\in[0,\pi/2]$ denotes the deviation angle
\begin{equation*}
\label{theta-cla}
\theta=\arccos \frac{\langle p-p_*,\omega\rangle}{|p-p_*|}.
\end{equation*}
Notice that one can restrict $\theta\in[0,\pi/2]$ by symmetrising 
$$B(|p-p_*|,\omega)=\frac{B(|p-p_*|,\omega)+B(|p-p_*|,-\omega)}{2}\mathbb{1}_{\theta\in[0,\pi/2]}.$$
The equation \eqref{uni:Boltzmann} can be written in the form of \eqref{uni-eq}
\begin{equation*}
    \d_t f+\frac{p}{m}\cdot\nabla_q f=-\frac14\d\aR^*_{\aB}\big(f,h'(f)\big),
\end{equation*}
where the dissipation potential $\aR^*_{\aB}$ is given by 
\begin{equation*}
\label{def:aR}
\aR^*_{\aB}(f,\av)=\int_{\R^d\times S^{d-1}}\Psi^*(\on \av)\Theta(f) B   \dd\omega \dd p_*.
\end{equation*}
In the case of \eqref{Q:phonon}, \eqref{Q:wave} and \eqref{Q:linear} the quantities $\Psi^*,\Theta(f)$ and $h(f)$ are given as in Table \ref{tab:models} and \eqref{H-all}.
In addition to the GENERIC building block \eqref{EL}-\eqref{MLL} associated to the free gradient $\ong$ given in Section \ref{sub-sec:B-gen}, the  equation \eqref{uni:Boltzmann} also has the following GENERIC building block $\{\aL, \d\aR^{*}_{\aB}, \aaE, \aS\}$  associated to $\on$, where $\aL,\aaE, \aS$  are given as in \eqref{EL} and \eqref{MLL}. More precisely, $\aaE$ and $\d\aR^{*}_{\aB}$ are given by
\begin{equation*}
    \label{block:Boltzmann}
\aaE(f)=\int_{\Do} \frac{|p|^2}{2m} f\quad\text{and}\quad\d \aR^*_{\aB}(f,\xi)=-\frac14\on\cdot\big(B \Theta(f) (\Psi^*)'(\on \xi)\big).
\end{equation*}




\subsection{Relativistic settings}
\label{sub-sec:Boltzmann-c}

 Let $m>0$ denote particle's mass at rest, and $c$  denote the speed of light. 
The energy of a relativistic particle with momentum $p$ is given by
\begin{align*}
e(p)=cp_0\quad\text{and}\quad p_0\defeq\sqrt{(mc)^2+|p|^2}.
\end{align*}

To distinguish them from the post-collisional momenta in the classical case \eqref{post-p}, we denote the relativistic post-collisional momenta by $\hat p'$ and $\hat p_*'$. We have the following parametrisation such that the momentum and energy conservation laws \eqref{sec4:cl} holds, see for instance \cite{HJ24,Str11}
\begin{equation}
\label{post-p-c}
    \begin{aligned}
     \hat p'&=\frac{p+p_*}{2}+\frac{g}{2}\Big(I_d+(\rho-1)\frac{(p+p_*)\otimes(p+p_*)}{|p+p_*|^2}\Big)\omega,\\
      \hat p_*'&=\frac{p+p_*}{2}-\frac{g}{2}\Big(I_d+(\rho-1)\frac{(p+p_*)\otimes(p+p_*)}{|p+p_*|^2}\Big)\omega
    \end{aligned}
\end{equation}
for some $\omega\in S^{d-1}$. For the sake of completeness, we verify the above parametrisation indeed satisfies the conservation laws in Lemma \ref{app-prop:post-p}.

We define the energy-momentum $(d+1)$-vector
    \begin{align*}
    p^\mu=(p_0,p)^T \quad\text{and}\quad p_\mu=(p_0,-p)^T \in \R^{d+1},  
    \end{align*}
    which satisfies the so-called on-shell condition
    \begin{equation*}
p^\mu \cdot p_\mu=p_0^2-|p|^2=(mc)^2.
    \end{equation*}
Let $g$ and $s$ denote the momentum and energy in the centre-of-mass framework given by
\begin{equation}
\label{s-g}    
\begin{gathered}
s=(p^\mu+p^\mu_*)\cdot  (p_\mu+(p_*)_\mu)=(p_0+p_{0*})^2-|p+p_*|^2,\\
g=\sqrt{-(p^\mu-p^\mu_*)\cdot  (p_\mu-(p_*)_\mu)}=\sqrt{-(p_0-p_{0*})^2+|p-p_*|^2}.
    \end{gathered}
    \end{equation}
    Notice that
    \begin{align*}
        s=4(mc)^2+g^2.
    \end{align*}
For the reason of completeness, we summarise the details of the centre-of-mass framework and Lorentz transformation in Appendix \ref{app:Lorentz}.

Similar to \eqref{def:on} in the non-relativistic setting, for $\phi=\phi(q,p)$, we define the Boltzmann relativistic discrete gradient $\onc$  by 
\begin{align}
\label{def:on-c}
\onc \phi=\hat\phi'+\hat \phi'_*-\phi-\phi_*,   
\end{align}
where we write
\begin{align*}
   \hat \phi'=\phi(q,\hat p')\quad\text{and}\quad \hat \phi_*'=\phi(q,\hat p_*').
\end{align*}
For any  $\phi=\phi(q,p)$ and $G=G(q,p,p_*,\hat p',\hat p_*')$, the following integration by parts formula holds  
\begin{align*}
\int_{\G\times S^{d-1}}G\onc\phi\dd\omega\dd p_*\dd p\dd q=-\int_{\Do}(\onc\cdot G) \phi\dd p\dd q,
\end{align*}
where the discrete divergence operator, $\onc\cdot G$, is given by
\begin{align*}
\onc\cdot G(q,p)=&\int_{\R^d\times S^{d-1}}G(q,p,p_*,\hat p',\hat p_*')+G(q,p_*,p,\hat p_*',\hat p')\\
&-\frac{\hat p_0'{\hat p_{0*}'}}{p_0p_{0*}}\big(G(q,\hat p',\hat p_*',p,p_*)+G(q,\hat p_*',\hat p',p_*,p)\big)\dd \omega\dd p_*.
\end{align*}
We note that $\frac{\hat p_0' {\hat p_{0*}'}}{p_0p_{0*}}$ is the Jacobian of the transformation $(p',p_*')\mapsto (p,p_*)$.

In \cite{Str11}, the following lemma to evaluate the momentum-energy Dirac measure $\delta$ in \eqref{eq:2-body operator} has been shown.
\begin{lemma}[\cite{Str11}, Theorem 2]\label{lem:cor-c}
Let $\cB^c=\cB^c(p,p_*,\hat p',\hat p_*')\ge0$ be the collision kernel in \eqref{eq:2-body operator}.
For any $G=G(q,p,p_*,w,w_*)$, we have 
\begin{align*}
    \int_{\Do} \delta \cB^cG(q,p,p_*,w,w_*)\frac{\dd w}{w_0}\frac{\dd w_*}{w_{0*}}=\int_{S^{d-1}}B^c G(q,p,p_*,\hat p',\hat p_*')\dd\omega,
\end{align*}
where $\hat p'$ and $\hat p_*'$ are given by \eqref{post-p-c}, and $B^c=B^c(p,p_*,\omega) \ge 0$ is the modified collision kernel such that 
\begin{align*}
\cB^c=\frac{2^{d-2}\sqrt s g^{2-d}}{p_0'p_{0*}'}B^c.
\end{align*}
\end{lemma}
Notice that the kernel $\cB^c$ satisfies the symmetrical condition \eqref{def:cB}, that is
\[
\cB^c(p,p_*,p',p_*')=\cB^c(p',p_*',p,p_*).
\]
However, this is not true for the kernel $B^c$. We note that $ \frac{\dd p\dd p_*}{p_0p_{0*}} =  \frac{\dd \hat p'\dd \hat p_*'}{\hat p_0'\hat p_{0*}'}$ is a Lorentz invariant measure.

The parametrised relativistic Boltzmann type of equations can be written as
\begin{equation}
    \label{uni:Boltzmann-c}
   \d_t f+\frac{cp}{p_0}\cdot \nabla_q f=Q^{\aB,c}(f).
\end{equation}
The relativistic (quantum) Boltzmann, four-wave kinetic, and linear Boltzmann collision operators have the form of \eqref{Q:phonon}, \eqref{Q:wave} and \eqref{Q:linear} associated with a relativistic kernel $B^c$. 
Let $\sigma^c,\,b:\R_+\to\R_+$.  The relativistic collision kernel $B^c$ is given by
\begin{equation}
\label{def:B-c}
B^c=v_c\sigma^c(g)b(\hat \theta),\quad v_c:=\frac{cg\sqrt s}{p_0p_{0*}},
\end{equation}
where $v_c$ is the so-called M\o ller velocity.
In the above, $\hat \theta\in[0,\pi/2]$ denotes the scattering angle
\begin{equation}
\label{theta-c}
\hat \theta=\arccos \frac{(p^\mu-p^\mu_*)\cdot(p'_\mu-(p_\mu)_*')}{g^2}\in[0,\pi/2].
\end{equation}
The equation \eqref{uni:Boltzmann-c} can be written in the form of \eqref{uni-eq}
\begin{equation*}
    \d_t f+\frac{cp}{p_0}\cdot\nabla_q f=-\frac14\d\aR^{c*}_{\aB}\big(f,h'(f)\big),
\end{equation*}
where the relativistic dissipation potential $\aR^{c*}_{\aB}$ is given by 
\begin{equation*}
\label{def:aR-c}
\aR^{c*}_{\aB}(f,\av)=\int_{\R^d\times S^{d-1}}B^c\Theta(f) \Psi^*(\onc \av)   \dd\omega \dd p_*.
\end{equation*}
In the case of relativistic (quantum) Boltzmann, wave kinetic and linear Boltzmann equations, the quantities $\Psi^*,\Theta(f)$ and $h(f)$ are given as in Table \ref{tab:models} and \eqref{H-all}.

In addition to the GENERIC building block \eqref{EL}-\eqref{MLL} associated to the free gradient $\ong$ given in Section \ref{sub-sec:B-gen}, the  equation \eqref{uni:Boltzmann-c} also has the following GENERIC building block $\{\aL, \d\aR^{c*}, \aaE, \aS\}$ associated to $\onc$, where $\aL,\aaE, \aS$  are given as in \eqref{EL} and \eqref{MLL}. More precisely, $\aaE$ and $\d\aR^{c*}$ are given by
\begin{equation*}
    \label{block:Boltzmann-c}
\aaE^c(f)=\int_{\Do} cp_0 f\quad\text{and}\quad \d \aR^{c*}(f,\xi)=-\frac14\onc\cdot\big(B^c \Theta(f) (\Psi^*)'(\onc \xi)\big).
\end{equation*}
\begin{remark}
 \label{sec:limit-n=2}
In view of the parametrisation presented in the previous subsections, we can formally derive the Newtonian limit from the relativistic Boltzmann equation to the classical one. In fact, from \eqref{s-g}, \eqref{def:B-c}, and \eqref{theta-c}, as  $c\to+\infty$, we have  \begin{align*}
\hat\theta\to\theta,\quad cp/p_0\to p/m,\quad g\to |p-p_*|\quad \text{and}\quad v_c\to \frac{2|p-p_*|}{m}.    
\end{align*} 
In addition, it follows from  \eqref{post-p-c} that the relativistic post-collision momenta converge to the classical ones defined in \eqref{post-p}. Then the relativistic kernels converges to the classical kernels
     \begin{align*}
        B^c=v_c\sigma^c(g)b(\hat \theta)\to B=\sigma(|p-p_*|)b(\theta),
    \end{align*}
with
\begin{align*}
   \sigma (|p-p_*|)= \frac{2|p-p_*|}{m}\sigma^c (|p-p_*|).  
    \end{align*}
We refer the reader to \cite{strain2010global,HJ24} for related papers on the Newtonian limit.
\end{remark}

\section{Small angle limit and Landau type of equations}\label{sec:Landau}

In this section, we perform the small-angle limits of the Boltzmann type equations presented in Section \ref{sec:Boltzmann} to derive the Landau-type equations shown in Figure \ref{fig:Landau}. This is motivated by the celebrated grazing limit from the classical Boltzmann equation to the classical Landau equation, see for instance \cite{villani1998new,carrillo2024landau}. Recently, this has been extended to the spatially homogeneous relativistic Boltzmann  \cite{HJ24}, the quantum Boltzmann equation \cite{GPTW25} and the 4-wave kinetic equations \cite{DH25C}. We consider this limit in a unified manner, in particular, as a consequence, the calculations for the relativistic quantum Boltzmann case is new. 

We consider the classical and relativistic collision kernels given by \eqref{def:B} and \eqref{def:B-c}
\begin{equation*}
    B(p,p_*,\omega)=\sigma(|p-p_*|)b(\theta),\quad  B^c(p,p_*,\omega)=v_c\sigma^c(g)b(\hat \theta).
\end{equation*}

We consider the singular 
angle function $\beta(\theta)\defeq\sin\theta ^{d-2}b(\theta)$ such that $\supp(\beta)\subset[0,\pi/2]$,
\begin{equation}
\label{beta:ass}
    \beta(\theta) = \sin\theta^{d-2}b(\theta) \gtrsim \theta^{-1-\nu}\quad\text{and}\quad \int_0^{\frac{\pi}{2}}\beta(\theta)\theta^2\dd\theta=8(d-1)/|S^{d-2}|
\end{equation}
for some $\nu\in(0,2)$.
The constant on the right-hand side is chosen to normalise $\beta$.
For $d=2$, we take $|S^0|=2$. 

For $\varepsilon\in (0,1)$, we take the following scaling of $\beta$
\begin{equation*}
\label{beta-eps}
\beta^\varepsilon(\theta)={\pi^3}/{\varepsilon^3}\beta\Big(\frac{\pi\theta}{\varepsilon}\Big)\quad\text{and}\quad \beta^\varepsilon(\theta)=\sin\theta ^{d-2}b^\varepsilon(\theta).
\end{equation*}
We define the scaling classical and relativistic collision operator by replacing $b$ by $b^\varepsilon$
\begin{equation}
\label{def:B-varepsilon}
    B_\varepsilon=\sigma(|p-p_*|)b^\varepsilon(\theta)\quad\text{and}\quad  B^c_\varepsilon=v_c\sigma^c(g)b^\varepsilon(\hat \theta).
\end{equation}
In this section, we  will study the small angle limit, i.e. as $\varepsilon\to0$ of the non-relativistic  Boltzmann equation \eqref{uni:Boltzmann} 
\begin{equation}
   \d_t f+\frac{p}{m}\cdot \nabla_q f=Q^{\aB}_\varepsilon(f), \label{grazing-limit}
\end{equation}
and of the relativistic quantum Boltzmann equation \eqref{uni:Boltzmann-c}
\begin{equation*}
  \d_t f+\frac{cp}{p_0}\cdot \nabla_q f=Q^{\aB,c}_\varepsilon(f). \label{grazing-limit-c}    
\end{equation*}
In the above, the rescaled collision operators $Q^{\aB}_\varepsilon(f)$ and $Q^{\aB,c}_\varepsilon(f)$
are obtained respectively from the collision operators $Q^{\aB}(f)$ and $Q^{\aB,c}(f)$ by replacing the kernel $B$ by the corresponding rescaled kernel $B_\varepsilon$ defined in \eqref{def:B-varepsilon}.


Next we will derive the limiting Landau-type systems in both non-relativistic and relativistic settings, then show that they are also GENERIC systems. To this end, we define the classical Landau gradient operator $ \widetilde \nabla$ by
\begin{align*}
   \widetilde \nabla f=\Pi_{(p-p_*)^\perp}(\nabla_p f-\nabla_{p_*}f_*),
\end{align*}
where $\Pi_{(p-p_*)^\perp}$ denotes the orthogonal projection onto $(p-p_*)^\perp$.
Let $G=G(q,p,p_*):\R^{3d}\to \R^d$. Then the following integration by parts formula holds
\begin{align*}
    \int_{\R^{3d}}G\cdot \widetilde \nabla \phi\dd p_*\dd p\dd q=-\int_{\R^{2d}}\widetilde\nabla\cdot G \phi\dd p\dd q,
\end{align*}
where the discrete Landau divergence operator $\widetilde\nabla\cdot G$ is given by 
\begin{align*}
\widetilde\nabla\cdot G(q,p)=\nabla_p\cdot\int_{\R^d}\Pi_{(p-p_*)^\perp}\big(G(q,p,p_*)-G(q,p_*,p)\big)\dd p_*.  \end{align*}
We also define the Landau kinetic kernels by
\begin{equation}
\label{def:sigma-bar}
 \overline{\sigma}(|p-p_*|)=\sigma(|p-p_*|)|p-p_*|^2\quad\text{and}\quad   v_c\overline \sigma^c(g)=\frac{cg\sqrt s}{p_0p_{0*}}\sigma^c(g)g^2.
\end{equation}
\subsection{Small-angle limit of the (quantum) Boltzmann equation in the non-relativistic setting}
\label{sub-sec:Landau}
To show the small-angle limit, we apply the following grazing limit lemma.    
\begin{lemma}
\label{lem:grazing}
  Let $\varepsilon\in(0,1)$. Let $\kappa(f)=a+\alpha f$ and $a,\alpha\in\{-1,0,1\}$. Let $\varepsilon \chi=\pi\theta$.
For any $f,\phi\in \cS(\R^{2d})$, we have, as $\varepsilon\to0$ 
\begin{multline*}
\int_{S^{d-2}_{k^\perp}}\kappa(f')\kappa(f_*')\overline\nabla \phi\\
\rightarrow \frac{\chi^2|S^{d-2}|}{8(d-1)}|p-p_*|^2\Big(2\big((\nabla_p-\nabla_{p_*})\kappa( f)\kappa( f_*)\big)\cdot \Pi_{(p-p_*)^\perp}\big(\nabla_p\phi-\nabla_{p_*}\phi_*\big)\\
\quad+\kappa( f)\kappa( f_*)(\nabla_p-\nabla_{p_*})\cdot\big(\Pi_{(p-p_*)^\perp}(\nabla_p \phi-\nabla_{p_*} \phi_*)\big)\Big),
\end{multline*}
where $S^{d-2}_{k^\perp}=\{p\in S^{d-1}\mid k\cdot p=0\}$. For $d=2$ and $k=(k_1,k_2)\in S^1$, we use the notation $S^{0}_{k^\perp}=\{(k_2,-k_1),\,(-k_2,k_1)\}$, and  $ \int_{S^{0}_{k^\perp}}f=f(k_2,-k_1)+f(-k_2,k_1)$. 
\end{lemma}
\begin{proof}
The proof of this lemma is a direct adaption of \cite[Lemma 3.2]{DH25C} for the kinetic wave equation by replacing $\kappa(f)=f$ there to $\kappa(f)=1+\alpha f$ with $\alpha\in\{-1,0,1\}$. Thus we omit it refer to \cite{DH25C} for the detail calculations.    
\end{proof}
To derive the small-angle limit of the (quantum) Boltzmann equations, we apply Lemma \ref{lem:grazing} by taking $\kappa(f)=1+\alpha f$. We recall that $Q^{\aB,\varepsilon}_\alpha (f)$ is the collision operator given by \eqref{Q:phonon} associated to the kernel $B^\varepsilon$ given in \eqref{def:B-varepsilon}. From the weak formulation, the identity 
\eqref{beta:ass} and Lemma \ref{lem:grazing}, we have 
\begin{align*}
\lim_{\varepsilon\to0}\langle Q^{\aB,\varepsilon}_\alpha (f),\phi\rangle&={}\lim_{\varepsilon\to0}\frac12\int_{\R^{3d}}\sigma ff_*\int_0^{\frac{\varepsilon}{2}} b^\varepsilon \int_{S^{d-2}_{k^\perp}}(1+\alpha f') (1+\alpha f_*')\overline\nabla \phi \\
&={}\frac12\int_{\R^{3d}}\sigma |p-p_*|^2\widetilde\nabla \phi\cdot\big(ff_*\widetilde\nabla \big((1+\alpha f)(1+\alpha f_*)\big)\\
&\quad -(1+\alpha f)(1+\alpha f_*)\widetilde\nabla(ff_*)\big)\\
&={}-\frac12\int_{\R^{3d}}\overline\sigma \widetilde\nabla \phi\cdot\Big(ff_*(1+\alpha f)(1+\alpha f_*) \wn \log\frac{f}{1+\alpha f}\Big)\\
&={}\langle Q_{\alpha}^{\aL} (f),\phi\rangle, 
\end{align*}
where we have used the property that $h'_\alpha(f)=\log \frac{f}{1+\alpha f}$ and the straightforwardly calculation
\begin{align*}
&ff_*(1+\alpha f)(1+\alpha f_*)\wn \log \frac{f}{1+\alpha f}\\
=&{}\Pi_{(p-p_*)^\perp}ff_*(1+\alpha f)(1+\alpha f_*)\Big(  \frac{\nabla_p f}{f(1+\alpha f)} -\frac{\nabla_{p_*} f_*}{f_*(1+\alpha f_*)}\Big) \\
=&{}\Pi_{(p-p_*)^\perp}\big( f_*(1+\alpha f_*) \nabla_p f -f(1+\alpha f) \nabla_{p_*} f_*\big) \\
=&{}\Pi_{(p-p_*)^\perp}\Big( f_*(1+\alpha f_*) \big((1+\alpha f)\nabla_p f-f\nabla_p(1+\alpha f)\big) \\
&\quad -f(1+\alpha f) \big((1+\alpha f_*)\nabla_{p_*} f_*-f_*\nabla_{p_*}(1+\alpha f_*)\big)\Big) \\
=&{}\Pi_{(p-p_*)^\perp}\Big( (1+\alpha f)(1+\alpha f_*) \wn(ff_*)-ff_*\wn\big((1+\alpha f)(1+\alpha f_*)\big)\Big).
\end{align*}
Hence, the operator $Q^{\aL}_\alpha$  is given by 
\begin{equation*}
\begin{aligned}
Q_{\alpha}^{\aL}(f)&=\frac12\widetilde \nabla\cdot \Big(\overline\sigma ff_*(1+\alpha f)(1+\alpha f_*) \wn \log\frac{f}{1+\alpha f}\Big)\\
&=\nabla_p\cdot \int_{\R^d}
\overline\sigma\Pi_{(p-p_*)^\perp}\big(f_*(1+\alpha f_*)\nabla_p f -f(1+\alpha f)\nabla_{p_*} f_* \big).\label{Q-phonon:Landau}
\end{aligned}
\end{equation*}
To derive the small-angle limit of the kinetic wave equation, we apply Lemma \ref{lem:grazing} with  $\kappa(f)=f$ and $\phi=h'_{wave}(f)=f^{-1}$. By similar computations as above, we obtain
\begin{equation*}
\lim_{\varepsilon\to0}\langle Q^{\aB,\varepsilon}_{wave}(f),\phi\rangle = \langle Q^{\aL}_{wave}(f),\phi\rangle,  
\end{equation*}
where the Landau wave collision operator $Q^{\aL}_{wave}(f)$ is given by
\begin{equation*}
Q^{\aL}_{wave}(f)=-\nabla_p\cdot\int_{\R^{d}}\overline\sigma(ff_*)^2\Pi_{(p-p_*)^\perp}\big(\nabla_v f^{-1}-\nabla_{p_*}f^{-1}_*\big) \dd p_*.\label{Q-wave:Landau}   
\end{equation*}
More details for this limit can be found in \cite{DH25C}. Finally, to derive the linear Landau equation, we apply  Lemma \ref{lem:grazing} by taking $\kappa=1$ and $\phi=f$, and obtain
\begin{equation*}
\lim_{\varepsilon\to0}\langle Q^{\aB,\varepsilon}_{linear}(f),\phi\rangle = \langle Q^{\aL}_{linear}(f),\phi\rangle,  
\end{equation*}
where the linear Landau operator $Q^{\aL}_{linear}(f)$ is given by
\begin{equation*}
Q^{\aL}_{linear}(f)=\nabla_p\cdot \int_{\R^d}
\overline\sigma\Pi_{(p-p_*)^\perp}\big(\nabla_p f -\nabla_{p_*} f_* \big),\label{Q-linear:Landau}    
\end{equation*}
In summary, we have shown that, in the small-angle limit in the non-relativistic setting, the quantum /wave/linear Boltzmann collision operators, $Q^{\aB,\varepsilon}_\alpha, Q^{\aB,\varepsilon}_{wave}$ and $Q^{\aB,\varepsilon}_{linear}$ converges respectively to the quantum/wave/linear Landau operators $Q^{\aL}_{\alpha}, Q^{\aL}_{wave}$ and $Q^{\aL}_{linear}$. Combining with appropriate compactness results, which we do not investigate here, we expect convergence of the dynamics, that is, using the common form, the Boltzmann-type equation \eqref{grazing-limit} converges to the non-relativistic Landau-type of equation
\begin{equation}
    \label{uni:Landau}
\d_t f+\frac{p}{m}\cdot \nabla_q f=Q^{\aL}(f),
\end{equation}
where $Q^{\aL}(f)$ is either $Q^{\aL}_{\alpha}, Q^{\aL}_{wave}$ or $Q^{\aL}_{linear}$. 
\subsection{GENERIC formulation of the non-relativistic Landau-type equation}
In this section, we will show that the limiting Landau-type equation \eqref{uni:Landau} can be cast into the GENERIC framework. In fact, the equation \eqref{uni:Landau} can be written in the form of \eqref{uni-eq}
\begin{equation*}
    \d_t f+\frac{p}{m}\cdot\nabla_q f=\frac12\d\aR_{\aL}^*\big(f,h'(f)\big),
\end{equation*}
where the dissipation potential $\aR^*_{\aL}$ is given by 
\begin{equation*}
\label{def:aR-L}
\aR^*_{\aL}(f,\av)=\int_{\R^d\times S^{d-1}}\overline\sigma \Theta_{\aL}(f) \Psi^*(\wn \av)   \dd\omega \dd p_*.
\end{equation*}
In the case of \eqref{Q:phonon}, \eqref{Q:wave} and \eqref{Q:linear} the quantities $\Psi^*,\Theta_{\aL}(f)$ and $h(f)$ are given as in Table \ref{tab:comptb-L} and \eqref{H-all}. Notice that the small-angle (grazing collision) limit of Boltzmann-type equations—in particular, the (quantum) Boltzmann equation associated with a non-quadratic ($\cosh$) GENERIC structure leads to Landau-type equations, which are associated only with a quadratic GENERIC one ($\Psi^*(r)=r^2/2$).


\begin{table}
\centering
\renewcommand{\arraystretch}{1.3}
\begin{tabular}{@{}lccccc@{}}
\toprule
Models & $\Psi^*(r)$ & $\Theta_{\aL}(f)$ &  $h'(f)$  \\
\midrule
  (Quantum) Landau 
  & $r^2/2$   & $ff_*(1+\alpha f)(1+\alpha f_*)$ & $\log \frac{f}{1+\alpha f}$   \\
 Wave Landau & $r^2/2$ &  $(ff_*)^2$ & $-f^{-1}$   \\
Linear Landau  & $r^2/2$ & $1$ & $f$  \\
\bottomrule
\end{tabular}
\caption{Weight function for Landau type of equations}
\label{tab:comptb-L}
\end{table}

The following weak formulation of $Q_{\aL}(f)$ holds 
\begin{align*}
\langle Q_{\aL}(f),\phi\rangle=-\frac{1}{2}\int_{\G}\overline\sigma \Theta_{\aL}\wn\varphi\cdot \wn \dd\cH(f) \dd p_*\dd p\dd q.
\end{align*}
Similar to the case of Boltzmann type of equations, the mass, momentum and energy conservation law \eqref{app-cl} holds for \eqref{uni:Landau}, since $\wn(1,p,e)=0$.
The following entropy identity holds at least formally
\begin{equation}
\label{cH-Landau}
\cH(f_t)-\cH(f_0)=-\int_0^t  \cD_{\aL}(f_s)\dd s \quad\forall t\in[0,T],
\end{equation}
where the entropy dissipation is given by 
\begin{equation*}
  \cD_{\aL}(f)=\frac{1}{2}\int_{\G}\overline\sigma \Theta_{\aL}|\wn h'(f)|^2\dd p_* \dd p\dd q\ge 0.
\end{equation*}
The Landau type of equations \eqref{uni:Landau} can be cast as GENERIC systems with the 
following building blocks $\{\aL,\d\aR^*_{\aL},\aaE,\aS\}$, where $\aL,\aaE, \aS$  are given as in \eqref{EL} and \eqref{MLL}. More precisely, $\aaE$ and $\d\aR^{*}_{\aL}$ are given by
\begin{equation*}
    \label{block:Landau}
\aaE(f)=\int_{\Do}\frac{|p|^2}{2m}f\quad\text{and}\quad\d \aR^*_{\aL}(f,\xi)=-\frac{1}{2}\wn\cdot\big(\overline\sigma\Theta_{\aL}(f)\wn \xi\big).
\end{equation*}

\medskip

\subsection{The small-angle limit in the relativistic setting}
\label{sub-sec:Landau-c}
The following lemma is the relativistic counterpart of Lemma \ref{lem:grazing} whose proof will be provided in Appendix \ref{app:Lorentz}.
\begin{lemma}
\label{lem:grazing-c}
  Let $\varepsilon\in(0,1)$. Let $\kappa(f)=a+\alpha f$ and $a,\alpha\in\{-1,0,1\}$. Let $\varepsilon \chi=\pi\hat \theta$.
For any $f,\phi\in \cS(\R^{2d})$, we have, as $\varepsilon\to0$ 
\begin{multline*}
\int_{S^{d-2}_{k^\perp}}\kappa(f')\kappa(f_*')\onc \phi \rightarrow \frac{\chi^2|S^{d-2}|}{8(d-1)}\Big(2g^2\big((\nabla_p-\nabla_{p_*})\kappa( f)\kappa( f_*)\big)\cdot S(p,p_*)\big(\nabla_p\phi-\nabla_{p_*}\phi_*\big)\\
\qquad+\kappa( f)\kappa( f_*)(\nabla_p-\nabla_{p_*})\cdot\big(g^2S(p,p_*)(\nabla_p \phi-\nabla_{p_*} \phi_*)\big)\Big).  
\end{multline*}
\end{lemma}
Using this lemma, by similar computations as in the non-relativistic setting, we obtain that the relativistic Boltzmann type of equations \eqref{uni:Boltzmann-c} converge to the following relativistic Landau type of equations
\begin{equation}
    \label{uni:Landau-c}
    \d_t f+\frac{cp}{p_0}\cdot \nabla_q f=Q^c_{\aL}(f).
\end{equation}
Here $Q^c_{\aL}(f)$ is either the relativistic (quantum) Landau, wave Landau, or linear Landau collision operator. They are given respectively by
\begin{align}
Q_{\alpha}^{\aL,c}(f)&=\nabla_p\cdot \int_{\R^d}
 v_c\overline \sigma^c S(p,p_*)\big(f_*(1+\alpha f_*)\nabla_p f -f(1+\alpha f)\nabla_{p_*} f_* \big)\dd p_*,\notag\\ 
Q^{\aL,c}_{wave}(f)
&=-\nabla_p\cdot\int_{\R^{d}} v_c\overline \sigma^c(ff_*)^2 S(p,p_*)\big(\nabla_v f^{-1}-\nabla_{p_*}f^{-1}_*\big) \dd p_*,\notag\\
Q^{\aL,c}_{linear}(f)&=\nabla_p\cdot \int_{\R^d}
 v_c\overline \sigma^c S(p,p_*)\big(\nabla_p f -\nabla_{p_*} f_* \big)\dd p_*,\notag
\end{align}
where the kernel $v_c\overline\sigma^c$ is given by \eqref{def:sigma-bar}.
\subsection{GENERIC formulation of the relativistic Landau equation}
We will formulate the limiting relativistic Landau equation \eqref{uni:Landau-c} into the GENERIC framework. We define the relativistic Landau gradient
\begin{align*}
\wnc \varphi=\Pi_{\hat k^\perp}  \widetilde\Lambda(\nabla_{  p}\varphi-\nabla_{  p_*}\varphi_*),
\end{align*}
where $\hat k= \frac{\widetilde  p-\widetilde  p_*}{|\widetilde  p-\widetilde  p_*|}\in S^{d-1}$ is defined as in \eqref{def:k-c}, $\widetilde  p$  and $\widetilde  p_*$ denote the Lorentz transformation of $p$ and $p_*$. Let $\Lambda\in \R^{(d+1)\times (d+1)}$ denote the matrix of Lorentz transformation
 \begin{align*}
\Lambda=\begin{pmatrix}
           \rho & -\rho v^T\\
           -\rho v & \widetilde\Lambda
       \end{pmatrix}\quad\text{and}\quad 
        \widetilde \Lambda\defeq I_d+(\rho-1)\frac{v\otimes v}{|v|^2}\in\R^{d\times d}.
    \end{align*}
    The definition of $\rho$ and $v$, and the details of Lorentz transformation can be found in Appendix \ref{app:Lorentz}.

For $\varphi=\varphi(q,p)$ and $G=G(q,p,p_*)\in\R^d$, the following integration by parts formula holds
\begin{align*}
\int_{\G}G\cdot\wnc\phi\dd\eta=-\int_{\Do}(\wnc\cdot G)\phi\dd p\dd q,
\end{align*}
where $\wnc\cdot G$ is given by 
\begin{align*}
\wnc\cdot G(q,p)=&\nabla_p\cdot\int_{\R^d} \widetilde\Lambda^T \Pi_{\hat k^\perp} \big(G(q,p,p_*)-G(q,p_*,p)\big) \dd p_*.
\end{align*}
We define 
\begin{equation}
    \label{def:S}
\begin{aligned}
&S(p,p_*)\defeq\widetilde\Lambda^T\Pi_{\hat k^\perp}\widetilde\Lambda\\
=&{}\Big[\big(\big(p^\mu\cdot (p_\mu)_*\big)^2-(mc)^4\big)I_d-(mc)^2(p\otimes p+p_* \otimes p_*)\\
&+\big(p^\mu\cdot (p_\mu)_*\big)^2(p\otimes p_*+p_*\otimes p)\Big]\Big(\big(p^\mu\cdot (p_\mu)_*\big)^2-(mc)^4\Big)^{-1}.
\end{aligned}
\end{equation}
The relativistic Landau equation \eqref{uni:Landau-c} can be written in the form of \eqref{uni-eq}
\begin{equation*}
    \d_t f+\frac{cp}{p_0}\cdot\nabla_q f=-\frac12\d\aR^{c*}_{\aL}\big(f,h'(f)\big),
\end{equation*}
where the relativistic dissipation potential $\aR^{c*}_{\aL}$ is given by 
\begin{equation*}
\label{def:aR-L-c}
\aR^{c*}_{\aL}(f,\av)=\int_{\R^d\times S^{d-1}}v_c\overline\sigma^c\Theta_{\aL}(f)\Psi^*(\wnc \av)   \dd\omega \dd p_*,
\end{equation*}
the quantities $\Psi^*,\Theta_{\aL}$ and $h(f)$ are given as in Table \ref{tab:comptb-L} and \eqref{H-all}.
The equation \eqref{uni:Landau-c} associated with the mass, momentum and energy conservation laws  \eqref{app-cl}, since $\wnc(1,p,e)=0$. The entropy identity \eqref{cH-Landau} holds at least formally with the relativistic entropy dissipation given by 
\begin{equation*}
  \cD_{\aL}^c(f)=\frac{1}{2}\int_{\G}v_c\overline\sigma^c \Theta_{\aL}(f)|\wnc h'(f)|^2\dd p_* \dd p\dd q\ge 0.
\end{equation*}
The relativistic Landau type of equations \eqref{uni:Landau-c} can be cast as GENERIC systems with the 
building blocks $\{\aL,\d\aR^{c*}_{\aL},\aaE,\aS\}$, where $\aL,\aaE, \aS$ are given as in \eqref{EL} and \eqref{MLL}. More precisely, $\aaE$ and $\d\aR^{c*}_{\aL}$ are given by
\begin{equation*}
\label{block:Landau-c}
\begin{aligned}
\aaE^c(f)=\int_{\Do}cp_0f\quad\text{and}\quad\d\aR^{c*}_{\aL}(f,\xi) =\frac{1}{2}\wnc\cdot\big(v_c\overline\sigma^c\Theta_{\aL}(f)\wnc \xi\big).
\end{aligned}
\end{equation*}
\appendix

\section{Lorentz transformation}\label{app:Lorentz}

In this appendix, we summarise the Lorentz transformation, see for instance \cite{Str11,HJ24} for more details, and provide useful lemmas for the relativistic Boltzmann and Landau equations.

For any given $p$ and $p_*\in\R^d$, we define the quantities
    \begin{align*}
       v=\frac{p+p_*}{p_0+p_{0*}}\in\R^d\quad\text{and}\quad  \rho=\frac{p_0+p_{0*}}{\sqrt s}\in\R.
    \end{align*}
The Lorentz transformation and its inverse are given by the operators $\Lambda$ and $\Lambda^{-1}:\R^{d+1}\to\R^{d+1}$ 
    \begin{equation}
    \label{lt}
     \begin{aligned}
      \Lambda=\begin{pmatrix}
           \rho & -\rho v^T\\
           -\rho v & \widetilde\Lambda
       \end{pmatrix}\quad\text{and}\quad \Lambda^{-1}=\begin{pmatrix}
           \rho & \rho v^T\\
           \rho v & \widetilde\Lambda
       \end{pmatrix},
    \end{aligned}
    \end{equation}
    where $\widetilde \Lambda\in \R^{d\times d}$ and given by 
    \begin{equation*}
    \label{def:wLA}
        \widetilde \Lambda\defeq I_d+(\rho-1)\frac{v\otimes v}{|v|^2}.
    \end{equation*}
We recall the definition of  momentum and energy in the centre-of-mass framework \eqref{s-g}      
\begin{gather*}
s=(p^\mu+p^\mu_*)\cdot  (p_\mu+(p_*)_\mu)=(p_0+p_{0*})^2-|p+p_*|^2,\\
g=\sqrt{-(p^\mu-p^\mu_*)\cdot  (p_\mu-(p_*)_\mu)}=\sqrt{-(p_0-p_{0*})^2+|p-p_*|^2}.
    \end{gather*}
We denote the Lorentz transformation of $p^\mu$ by 
\begin{equation}
\label{p-mu}
(\widetilde  p_0,\widetilde  p)^T:=\widetilde  p^\mu:=\Lambda p^\mu.   
\end{equation}
We define
\begin{align}
\label{def:k-c}
    \hat k\defeq \frac{\widetilde  p-\widetilde  p_*}{|\widetilde  p-\widetilde  p_*|}\in S^{d-1}.
\end{align}
The Lorentz transformation of $p^\mu$ and $p^{\mu}_*$ are given by the following lemma.
    \begin{proposition}
    \label{app-lem:lorentz}
        For any $p^\mu,\, p^\mu_*\in \R^{d+1}$, we have 
        \begin{equation*}
        \widetilde  p^\mu=(\sqrt s/2,g\hat k/2)^T\quad \text{and}\quad \widetilde  p^\mu_*=(\sqrt s/2,-g\hat k/2)^T. 
        \end{equation*}
    \end{proposition}
       \begin{proof}
We first show that 
\begin{equation*}
    \label{tilde-p}
    \widetilde  p+\widetilde  p_*=0\quad\text{and}\quad |\widetilde  p|=|\widetilde  p_*|=\frac{g}{2}.
\end{equation*}
By definition, we have
\begin{align*}
 \widetilde  p=p+\Big(-\rho p_0+(\rho-1)\frac{v\cdot p}{|v|^2}\Big)v\quad\text{and}\quad
 \widetilde  p_*=p_*+\Big(-\rho p_{0*} +(\rho-1)\frac{v\cdot p_*}{|v|^2}\Big)v,
\end{align*}
and we have
\begin{align*}
 \widetilde  p+\widetilde  p_*
 &=(p+p_*)+\Big(-\rho (p_0+p_{0*}) +(\rho-1)\frac{(p_0+p_{0*})v\cdot v}{|v|^2}\Big)v\\
 &=(p+p_*)-(p_0+p_{0*})v=0.
\end{align*}
By using of the identities $v\cdot (p-p_*)=p_0-p_{0*}$ and $1-|v|^2
=\rho^{-2}$, we have 
\begin{align*}
 \widetilde  p-\widetilde  p_*&=(p-p_*)+\Big(-\rho (p_0-p_{0*}) +(\rho-1)\frac{v\cdot (p-p_*)}{|v|^2}\Big)v\\
 &=(p-p_*)+(p_0-p_{0*})\Big(-\rho +\frac{\rho-1}{|v|^2}\Big)v,
\end{align*}
and 
\begin{align*}
 |\widetilde  p-\widetilde  p_*|^2&=|p-p_*|^2+(p_0-p_{0*})^2|v|^2(-\rho+(\rho-1)|v|^{-2})^2\\
 &\quad +2(p_0-p_{0*})^2(-\rho+(\rho-1)|v|^{-2})\\
 &=|p-p_*|^2+(p_0-p_{0*})^2(\rho^2|v|^2+\frac{\rho^2-1}{|v|^2}-2\rho^2)\\
  &=|p-p_*|^2+(p_0-p_{0*})^2(\rho^2|v|^2-\rho^2)\\
  &=|p-p_*|^2-(p_0-p_{0*})^2=g^2.
\end{align*}
Hence, we have $\widetilde  p=g\hat k/2$ and $\widetilde  p_*=-g\hat sk/2$.

Similarly, we have $\widetilde  p_0+\widetilde  p_{0*}=\sqrt s$ and $\widetilde  p_0-\widetilde  p_{0*}=0$. Hence, we have $\widetilde  p_0=\widetilde  p_{0*}=\sqrt s/2$.
\end{proof}

We derive the post-collision momenta \eqref{post-p-c} that fulfil the momentum and energy conservation laws.
\begin{proposition}
    \label{app-prop:post-p}
    For pre- and post-collision status $p^\mu,\,p^\mu_*$ and $(\hat p^\mu),\,(\hat p^\mu)'_*$ satisfying the following momentum and energy conservation laws
\begin{equation}
    \label{app-cl}
p+p_*=\hat p'+\hat p_*'\quad\text{and}\quad p_0+p_{0*}=\hat p_0'+\hat p_{0*}',
\end{equation}
the post-collision momentum $p',\,p_*'$ are given by \eqref{post-p-c}
\begin{equation*}
    \begin{aligned}
     \hat p'&=\frac{p+p_*}{2}+\frac{g}{2}\Big(I_d+(\rho-1)\frac{v\otimes v}{|v|^2}\Big)\omega,\quad 
      \hat p_*'=\frac{p+p_*}{2}-\frac{g}{2}\Big(I_d+(\rho-1)\frac{v\otimes v}{|v|^2}\Big)\omega
    \end{aligned}
\end{equation*}
for some $\omega\in S^{d-1}$. The  post-collision energy $\hat p_0'$ and $\hat p_{0*}'$ are given by
\begin{equation*}
\label{ee}
    \begin{aligned}
    \hat p_0'=\frac{p_0+p_{0*}}{2}+\frac{g}{2}\frac{p+p_*}{\sqrt s}\cdot \omega\quad\text{and}\quad
     \hat p_{0*}'=\frac{p_0+p_{0*}}{2}-\frac{g}{2}\frac{p+p_*}{\sqrt s}\cdot \omega.
    \end{aligned}
\end{equation*}

\end{proposition}
    \begin{proof}
By conservation laws \eqref{app-cl}, we have \begin{align*}
    \quad \hat s'=s,\quad \hat g'=g\quad \text{and}\quad \hat v'=v,\quad \hat \rho'=\rho,\quad \Lambda^{\pm1}=(\hat \Lambda')^{\pm1}.
\end{align*}
We denote $\widetilde  p',\,\widetilde  p'_{0}$ and $\widetilde  p_*',\,\widetilde  p'_{0,*}$ the Lorentz transformation of the post-collision momenta and energy. Applying Proposition
    \ref{app-lem:lorentz} to $(\hat p^\mu)'$ and $(\hat p^\mu)'_*$, we have  
\begin{equation*}
\label{tilde-p'}
    \widetilde  p'+\widetilde  p'_*=0,\quad |\widetilde  p'|=|\widetilde  p'_*|=\frac{g}{2}\quad\text{and}\quad \widetilde  p_0'=\widetilde  p_{0*}'=\frac{\sqrt s}{2}.   
\end{equation*}
Then there exists $\omega\in S^{d-1}$ such that
\begin{align*}
    \widetilde  p'=\frac{g}{2} \omega\quad\text{and}\quad \widetilde  p'_*=-\frac{g}{2} \omega.
\end{align*}
Applying the inverse Lorentz transformation, we have  
\begin{align*}
    (\hat p_0',\hat p')^T&=\Lambda^{-1}(\sqrt s/2,g\omega/2)^T\\
    &=\Big(\frac{\sqrt s \rho }{2}+\rho v\cdot\omega\frac{g}{2},\frac{\sqrt s \rho v}{2}+\frac{g}{2}\Big(I_d+(\rho-1)\frac{v\otimes v}{|v|^2}\Big)\omega\Big)^T\\
    &=\Big(\frac{p_0+p_{0*}}{2}+\frac{g}{2}\frac{p+p_*}{\sqrt s}\cdot \omega,\frac{p+p_*}{2}+\frac{g}{2}\Big(I_d+(\rho-1)\frac{v\otimes v}{|v|^2}\Big)\omega\Big)^T,
\end{align*}
and $(\hat p_{0*}',\hat p_*')^T$ follows analogously.

\end{proof}

We show another representation of  the scattering angle defined in \eqref{theta-c}
\begin{align}
\label{app:theta-c}
    \hat \theta=\arccos \frac{(p^\mu-p^\mu_*)\cdot(p'_\mu-(p_\mu)_*')}{g^2}
\end{align}

\begin{proposition}
\label{app-prop:theta-c}
    The scattering angle defined in \eqref{app:theta-c} is equivalent to 
    \begin{equation*}
\label{RB-theta}
\hat \theta= \arccos \hat k \cdot\omega.
\end{equation*}
\end{proposition}
\begin{proof}
We note that
\begin{align*}
 \hat k\cdot\omega&= g^{-1}\Big((p-p_*)+\big(-\rho (p_0-p_{0*})v +(\rho-1)\frac{v\cdot (p-p_*)}{|v|^2}\big)v\Big)\cdot\omega  \\
 &= -g^{-1}(p_0-p_{0*})\rho v\cdot\omega+g^{-1}\big(I_d+\frac{v\otimes v}{|v|^2}\big)(p-p_*)\cdot\omega. 
\end{align*}
We recall the definition of the scattering angle \eqref{theta-c}
\begin{align*}
\theta&=\arccos \frac{(p^\mu-p^\mu_*)\cdot(p'_\mu-(p_\mu)_*')}{g^2}\\
&=\arccos \frac{(p_0-p_{0*})(p_0'-p_{0*}')-(p-p_*)(p'-p_*')}{g^2}.
\end{align*}
By direct calculation, we have 
\begin{align*}
&(p_0-p_{0*})(p_0'-p_{0*}')-(p-p_*)(p'-p_*')\\
=&{}\frac{g}{\sqrt s}(p_0-p_{0*})(p+p_*)\cdot \omega-g(I_d+(\rho-1)\frac{v\otimes v}{|v|})(p-p_*)\cdot \omega\\
=&{}g(p_0-p_{0*})\rho v\cdot \omega-g(I_d+(\rho-1)\frac{v\otimes v}{|v|})(p-p_*)\cdot \omega.
\end{align*}

\end{proof}
\begin{proposition}
The operator $S(p,p_*)$ given by \eqref{def:S}
\begin{align*}
S(p,p_*)&=\Big[\big(\big(p^\mu\cdot (p_\mu)_*\big)^2-(mc)^4\big)I_d-(mc)^2(p\otimes p+p_* \otimes p_*)\\
&\quad+\big(p^\mu\cdot (p_\mu)_*\big)^2(p\otimes p_*+p_*\otimes p)\Big]\Big(\big(p^\mu\cdot (p_\mu)_*\big)^2-(mc)^4\Big)^{-1},
\end{align*}
is equivalent to 
\begin{align*}
\label{S-2}
   S(p,p_*)=\widetilde \Lambda \Pi_{\hat k^\perp} \widetilde \Lambda,\quad \widetilde \Lambda=I_d+(\rho-1)\frac{v\otimes v}{|v|^2}.
\end{align*}
   
\end{proposition}
\begin{proof}
One can check straightforwardly by using the following equalities
    \begin{align*}
s=2\big((p^\mu\cdot (p_\mu)_*+(mc)^2\big) \quad\text{and}\quad
g^2=2\big((p^\mu\cdot (p_\mu)_*-(mc)^2\big),
\end{align*}
\end{proof}
The rest of this appendix is devoted to show the small-angle limit in Lemma \ref{lem:grazing-c}, that is to show 
\begin{equation}
\label{app:grazing}
\begin{aligned}
&\lim_{\hat \theta\to0}\hat\theta^{-2}\int_{S^{d-2}_{k^\perp}}\kappa(f')\kappa(f_*')\onc \phi \\
=&{}\frac{|S^{d-2}|}{8(d-1)}\Big(2g^2\big((\nabla_p-\nabla_{p_*})\kappa( f)\kappa( f_*)\big)\cdot S(p,p_*)\big(\nabla_p\phi-\nabla_{p_*}\phi_*\big)\\
&\quad+\kappa( f)\kappa( f_*)(\nabla_p-\nabla_{p_*})\cdot\big(g^2S(p,p_*)(\nabla_p \phi-\nabla_{p_*} \phi_*)\big)\Big),
\end{aligned}
\end{equation}
where $\kappa(f)=a+\alpha f$ and $a,\alpha\in\{-1,0,1\}$. The operator $\onc$ and $S(p,p_*)$ are defined as in \eqref{def:on-c} and \eqref{def:S} respectively. 

\begin{proof}[Proof of Lemma \ref{lem:grazing-c}]
We recall the Lorentz transformation \eqref{p-mu} of $p^\mu$ that $(\widetilde  p_0,\widetilde  p)^T=\Lambda (p_0, p)^T$. We use the notation $\widetilde\phi$ for the function $\widetilde\phi$ such that 
\begin{align*}
   \phi(q,p)=\widetilde \phi(q,\widetilde  p)\quad \text{for all}\quad (q,p)\in\Do.
\end{align*}
Let $\widetilde  p'$ and $\widetilde  p_*'$ denote the Lorentz transformation of $\hat p'$ and $\hat p_*'$. We write 
\begin{align*}
 \widetilde \phi'=\phi(q,\widetilde  p'),\quad \widetilde \phi'_*=\phi(q,\widetilde  p'_*),\quad \widetilde \phi=\phi(q,\widetilde  p),\quad \widetilde \phi_*=\phi(q,\widetilde  p_*). \end{align*}
Then to show \eqref{app:grazing}, we first show that 
\begin{equation}
\label{app:grazing-tilde}
\begin{aligned}
&\lim_{\hat \theta\to0}\frac{1}{\hat \theta^2}\int_{S^{d-2}_{ k^\perp}}\kappa(\widetilde f')\kappa(\widetilde f_*')\big(\widetilde \varphi'-\widetilde \varphi+\widetilde \varphi_*'-\widetilde \varphi_*\big) \\
=&{}\frac{|S^{d-2}|}{8(d-1)}\Big(2g^2\big((\nabla_{\widetilde  p}-\nabla_{\widetilde  p_*})\kappa( \widetilde  f)\kappa(\widetilde  f_*)\big)\cdot \Pi_{\hat k^\perp}\big(\nabla_{\widetilde  p}\widetilde  \phi-\nabla_{\widetilde  p_*}\widetilde  \phi_*\big)\\
&\quad+\kappa(\widetilde  f)\kappa(\widetilde  f_*)(\nabla_{\widetilde  p}-\nabla_{\widetilde  p_*})\cdot\big(g^2\Pi_{\hat k^\perp}(\nabla_{\widetilde  p} \widetilde \phi-\nabla_{\widetilde  p_*} \widetilde \phi_*)\big)\Big).
\end{aligned}
\end{equation}
Then \eqref{app:grazing} holds as a consequence of the following identity
  \begin{align}
  \label{id-nabla-tilde}
   \Pi_{\hat k^\perp}\big(\nabla_{\widetilde  p} -\nabla_{\widetilde  p_*}\big)\widetilde G=\Pi_{\hat k^\perp}\widetilde \Lambda\big(\nabla_p-\nabla_{p_*}\big)G.
  \end{align}
  for all $G=G(q,p,p_*)\in C^\infty(\R^{3d})$. We show \eqref{id-nabla-tilde} by the definition of Lorentz transformation \eqref{lt}
  \begin{gather*}
    \nabla_{\widetilde  p}\widetilde G= \Big(\frac{\d p}{\d \widetilde  p}\Big)^T\nabla_{p}G=\Big(\widetilde\Lambda+\frac{\rho}{\widetilde  p_0}v\otimes \widetilde  p\Big)^T\nabla_pG,\\
    \big(\nabla_{\widetilde  p} -\nabla_{\widetilde  p_*}\big)\widetilde G
=\widetilde  \Lambda\big(\nabla_p-\nabla_{p_*}\big)G+\frac{\rho g}{\sqrt s}\big((\nabla_{p} G+\nabla_{p_*}G)\cdot v\big)\hat k.
  \end{gather*}
We are left to show the limit \eqref{app:grazing-tilde}. 
By Proposition \ref{app-prop:post-p}, we have 
\begin{align}
\label{app-diff}
\widetilde  p'-\widetilde  p=\frac{g}{2}(\omega-\hat k)\quad \text{and}\quad \widetilde  p'_*-\widetilde  p_*=-\frac{g}{2}(\omega-\hat k).   
\end{align}
Notice that the difference \eqref{app-diff} is coincidence with the classical case \eqref{post-p} with $|p-p_*|$ replaced by $g$ and $k$ replaced by $\hat k$. For the reason of completeness, we sketch the proof here.

By Taylor expansion, we have
\begin{align*}
    &\onc \widetilde \phi=\frac{g}{2}(\omega-\hat k)\cdot\big(\nabla_{\widetilde  p}\widetilde \phi-\nabla_{\widetilde  p_*}\widetilde \phi_*\big)+\frac{g^2}{4}(\omega-\hat k)\otimes (\omega-\hat k):T,\\
    & \kappa(\widetilde  f')\kappa(\widetilde  f_*')=\widetilde \kappa\widetilde \kappa_*+\frac{g}{2}(\omega-\hat k)\cdot \big(\nabla_{\widetilde  p}-\nabla_{\widetilde  p_*}\big)\widetilde \kappa\widetilde  \kappa_*+O(|w-\hat k|^2),
\end{align*}
where $\widetilde  \kappa=\kappa(\widetilde f)$ and  $\widetilde  \kappa_*=\kappa(\widetilde f_*)$ and $T$ is given by 
\begin{align*}
T&=\int_0^1t\int_0^1\Big(D^2_{\widetilde  p}\widetilde \varphi\big(q,\frac{g}{2}s(t\omega+(1-t)k)+(1-s)k\big)\\
&+D^2_{\widetilde  p}\widetilde \varphi\big(q,-\frac{g}{2}s(t\omega+(1-t)k)+(1-s)k\big)\Big)\dd s\dd t.
\end{align*}
We recall the definition of the scattering angle \eqref{app:theta-c} 
\begin{equation*}
\label{omega-k}
\omega-\hat k=\hat k(\cos\hat \theta-1)+\gamma\sin\hat \theta,   \quad \gamma\in S^{d-2}_{\hat k^\perp},
\end{equation*}
where $(\cos\hat \theta-1)=-\frac{\hat \theta^2}{2}+o(\hat\theta^2)$ and $\sin\hat \theta=\hat \theta+o(\hat \theta)$. 

By classical argument as in \cite{villani1998new}, we have
 \begin{equation}
    \label{int-1}
\begin{aligned}
&\lim_{\hat\theta\to0}\hat \theta^{-2}\int_{S^{d-2}_{\hat k^\perp}}\onc\phi=\frac{|\widetilde  p-\widetilde  p_*|^2|S^{d-2}|}{8(d-1)}(\nabla_{\widetilde  p}-\nabla_{\widetilde  p_*})\cdot\Pi_{\hat k^\perp}\big(\nabla_{\widetilde  p}\widetilde \phi-\nabla_{\widetilde  p_*}\widetilde \phi_*\big).
\end{aligned}
\end{equation}
By using \cite[Proposition 3.2]{DGH25}
\begin{equation*}
\begin{aligned}
\int_{S^{d-2}_{k^\perp}} (\omega-\hat k)\otimes(\omega-\hat k) =\frac{|S^{d-2}|}{d-1}\Pi_{\hat k^\perp}+o(\hat\theta^2),
\end{aligned}
\end{equation*}
we have 
\begin{equation*}
\label{int-2}
\begin{aligned}
 &\int_{S^{d-2}_{\hat k^\perp}}   (\omega-\hat k)\cdot \big(\nabla_{\widetilde  p}-\nabla_{\widetilde  p_*}\big)\widetilde  \kappa \widetilde \kappa_*
(\omega-\hat k)\cdot\big(\nabla_{\widetilde  p}\widetilde \phi-\nabla_{\widetilde  p_*}\widetilde \phi_*\big)\\
 =&{}\int_{S^{d-2}_{\hat k^\perp}}   (\omega-\hat k)\otimes (\omega-\hat k):\big(\nabla_{\widetilde  p}-\nabla_{\widetilde  p_*}\big)\widetilde  \kappa \widetilde \kappa_* \otimes \big(\nabla_{\widetilde  p}\widetilde \phi-\nabla_{\widetilde  p_*}\widetilde \phi_*\big)\\
  =&{}\frac{|S^{d-2}|\hat \theta^2}{(d-1)}\big(\nabla_{\widetilde  p}-\nabla_{\widetilde  p_*}\big)\widetilde  \kappa \widetilde \kappa_* \cdot\Pi_{\hat k^\perp}\big(\nabla_{\widetilde  p}\widetilde \phi-\nabla_{\widetilde  p_*}\widetilde \phi_*\big)
\end{aligned}
\end{equation*}

\end{proof}
\subsection*{Acknowledgements}
M. H. D is funded by an EPSRC Standard Grant EP/Y008561/1. Z.~H. is funded by the Deutsche Forschungsgemeinschaft (DFG, German Research Foundation) – Project-ID 317210226 – SFB 1283.


\printbibliography
\end{document}